\documentclass[a4paper,12pt]{article}

\usepackage{fullpage}
\usepackage{amsthm,amsmath, amssymb}
\usepackage{palatino,comment}
\usepackage[usenames, dvipsnames]{color}
\usepackage[active]{srcltx}   
\input{xy}
\xyoption{all}
\xyoption{matrix}

\newtheorem{theorem}{Theorem}[section]
\newtheorem{teo}[theorem]{Theorem}

\newtheorem{lem}[theorem]{Lemma}

\newtheorem{prop}[theorem]{Proposition}

\newtheorem{coro}[theorem]{Corollary}

\theoremstyle{definition}

\newtheorem{defi}[theorem]{Definition}

\newtheorem{ex}[theorem]{Example}

\newtheorem{rem}[theorem]{Remark}

\DeclareMathOperator{\Hom}{Hom}
\DeclareMathOperator{\HOM}{HOM}

\DeclareMathOperator{\Ker}{Ker}
\def\Im{\mathrm{Im\ }}

\DeclareMathOperator{\id}{id}

\newcommand{\ot}{\otimes}

\newcommand{\wt}{\widetilde}

\def\ra{\overline}
\def\ul{\underline}

\def\A{\mathcal A}
\def\YD{\mathcal {YD}}
\def\I{\mathcal{I}}

\def\BB{\mathcal B}

\def\S{\mathcal S}
\def\D{\mathcal D}
\def\M{\mathcal M}
\def\uM{\underline{\mathcal M}}
\def\K{\mathcal K}

\def\Hom{\mathrm{Hom}}

\def\g{\mathfrak g}

\def\b{\mathfrak b}

\def\B{\mathfrak B}

\def\m{\mathfrak m}
\def\um{\underline{\mathfrak m}}

\def\H{\mathcal H}

\def\Z{\mathbb Z}

\def\N{\mathbb N}
\def\O{\mathcal O}

\def\om{\omega}
\def\To{\Rightarrow}

\def\pf{\begin{proof}}
\def\epf{\end{proof}}
\def\Int{\Lambda}

\def\soc{\mathrm{soc}}
\def\gr{\mathrm{gr}}

\title{Hopfological algebra for infinite dimensional Hopf	algebras}
\author{Marco A. Farinati 
\thanks{Partially supported by  UBACyT 2018-2021, 256BA.
Research member of CONICET - 
I.M.A.S. - Depto. de Matem\'atica,
 F.C.E.y N. 
Universidad de Buenos Aires, 
Ciudad Universitaria Pabell\'on I 
(1428) C.A.B.A.  Argentina, 
 \textit{E-mail address:} \texttt{mfarinat@dm.uba.ar}
}
}

\begin{document}

\maketitle
\noindent 2010 {\em Mathematics Subject Classification:}\,  16T05
16E35, 18G99, 18D99, 19A49, 81R50 
\\
  {\em Keywords:}  Co-Frobenius Hopf Algebras,
Tensor Triangulated Categories,  Homology Theories, $K_0$, 
 Categorification.

\begin{abstract}
We consider "Hopfological" techniques as in 
\cite{Ko} but for infinite dimensional Hopf algebras,
under the assumption of being co-Frobenius. In particular,
$H=k[\Z]\#k[x]/x^2$ is the first example,
 whose corepresentations category is d.g. vector
 spaces. Motivated by this example we define the 
"Homology functor" (we prove  it is homological) 
for any co-Frobenius algebra, with coefficients in $H$-comodules,
that recover usual homology of a complex when $H=k[\Z]\#k[x]/x^2$.
 Another easy example of co-Frobenius Hopf algebra 
 gives the category of  "mixed complexes" and we see (by computing an 
example) that this homology
theory differs from cyclic homology, although there exists a long exact sequence
analogous to the SBI-sequence. Finally, because we have a tensor triangulated category,
its $K_0$ is a ring, and we prove a "last part of a localization exact sequence" for $K_0$
 that allows us to compute -or describe- $K_0$ of some family of examples,
giving light of what kind of rings can be categorified using this techniques.
\end{abstract}


\section*{Introduction}
This paper has mainly 3 contributions:

(1) The "Hopfological algebra" can be developed not only for finite dimensional 
Hopf algebras but also for infinite dimensional ones, provided they are co-Frobenius.
The language of comodules is better addapted than the language of modules.

(2) The formula "$\Im d / \Ker d$" can written in Hopf-co-Frobenius language.

(3) Some  $K$-theoretical results allow us to compute $K_0$ of the 
stable categories associated to  co-Frobenius Hopf algebras of the form $H_0\#\B$, with
$H_0$ cosemisimple and $\B$ finite dimensional.

\

The paper is organized as follows: 
In Section 1 we exhibit point (1) above, in Section 2 we make point (2).
In Section 3 we develop some tools to understand the triangulated structure.
 In Section 4 we exhibit the first examples.
 Section 5 deals with $K_0$. Section 6
 illustrate the first step on how to develop -in the setting 
of co-Frobenius Hopf algebras- the
direction taken in \cite{Qi} for finite dimensional Hopf algebras.

\subsection*{Acknowledgments} I wish to thank Juan Cuadra
for answering  questions and pointing useful references  
on co-Frobenius coalgebras. I also thanks Gast\'on A. Garc\'ia
for helping me with coradical problems. This work was partially supported by
UBACyT 2018-2021, 256BA.

\section{Integrals, Co-Frobenius  and triangulated structure}

$k$ will be a field, $H$ a Hopf algebra over $k$, all comodules are right comodules.
The category of comodules is denoted $\M^H$ and the subcategory of
finite dimensional comodules is denoted $\m^H$.

\subsection{Integrals}
\begin{defi}
(Hochschild, 1965; G. I. Kac, 1961; Larson-Sweedler,
1969).  A  (left) integral is a linear map
$\Int:H\to k$ such that
\[
(\id\ot \Int)\Delta h=\Int(h) 1 \ \forall h\in H
\]
that is, $h_1\Int(h_2)=\Int(h)1$.
\end{defi}
It is well-known that the dimension of the space of (left) integral is
 $\leq 1$. In case $H$ admits a non-zero (left) integral $\Int\in H^*$, $H$ will be called
co-Frobenius. 
The following is well-known, we refer
to \cite{ACE} and \cite{AC} and references therein for the proofs:
\begin{teo}If $H$ is co-Frobenius then, in the category of (say right) $H$-comodules
\begin{enumerate}
\item there exists enough projectives;
\item every finite dimensional comodule is a quotient of a finite dimensional
projective, and embeds into a finite dimensional injective;
\item being projective is the same as being injective.
\end{enumerate}
\end{teo}
If $M$ and $N$ are two objects in a category $\BB$, denote
$\Hom_\BB(M,N)$ set of morphisms and
$I_\BB(M,N)$ the subset of
$\Hom_\BB(M,N)$ consisting on morphisms
 that factors through an injective object of $\BB$. Denote
\[
\ul\Hom_\BB(M,N):=\frac{\Hom_\BB(M,N)}{I_\BB(M,N)}
\]
The category whose objects are $H$-comodules and morphism
$\ul\Hom^H$ is called the {\em stable} category and it is denoted
$\ul\M^H$. Similarly $\ul\m^H$ is the stable category associated to $\m^H$.
 By the above theorem, $\ul\m^H$ is embedded fully faithfully in 
$\ul\M^H$.
With these preliminaries, one can prove the following main construction:
\begin{teo} 
If $H$ is a co-Frobenius Hopf algebra then 
$\uM^H$  has a natural structure of triangulated category,
$\um^H$ is a triangulated subcategory.
\end{teo}

\begin{proof}
We apply directly Happel's Theorem 2.6  of \cite{Ha}. The only thing to do is to notice that 
$\M^H$ (and $\m^H$) are Frobenius exact categories. Using Happel's notation,
let  $\BB$ be
 an additive category embedded as a full and
extension-closed subcategory in some abelian category $ \A$, 
 and $\S$ the set of short exact sequences in
 $\A$ with terms in $\BB$.
 For as, since both 
$\M^H$ and $\m^H$ are already abelian,   we have $\A=\BB$ and the notion
of $\S$-projective and $\S$-injective is the same as usual projectives and injectives.
Maybe we just remark that an object in $\m^H$ is injective in $\m^H$ if and only if it is
 injective in $\M^H$ and similarly for projectives (see Lemma \ref{lemaproy} as an 
illustration). 

An exact category $(\BB,\S)$ is called a Frobenius category if
$(\BB,\S)$ has enough $\S$-projectives and enough $\S$-injectives and if moreover
the $\S$-projectives coincide with the $\S$-injectives. In our case,
$\BB=\M^H$ or $\BB=\m^H$ are clearly Frobenius categories if
$H$ is a co-Frobenius Hopf algebra. 
Theorem 2.6 in \cite{Ha} just state that the stable category 
$\ul\BB$ is triangulated. 
\end{proof}

\begin{lem}
\label{lemaproy}
If $P\in\m^H$ then $P$ is projective in $\m^H$ if and only if 
it is projective in $\M^H$.
\end{lem}
\begin{proof}
If $P$ is projective in $\M^H$ then then it has the lifting property for all comodules, in
 particular for the finite dimensional ones. Assume $P$ is projective in $\m^H$ and consider a
 diagram of comodules
\[
\xymatrix@-1ex{
Z\ar@{->>}[r]^\pi&Y\\
&P\ar[u]^f
}\]
where $Z$ and $Y$ are not necessarily finite dimensional.
Since $P$ is finite dimensional, one can consider
 $Y'=f(P)\subseteq Y$, clearly $Y'$ is a finite dimensional comodule,
with generators say $\{y_1,\dots,y_n\}$. Since $\pi$ is surjective, 
one may found $z_i$ ($i=1,\dots,n$) with
$\pi(z_i)=y_i$ and {\em there exists a finite dimensional} subcomodule $Z'\subseteq Z$
containing all $z_i$'s, hence we have a diagram
\[
\xymatrix{
Z'\ar@{->>}[r]^{\pi|_{Z'}}&Y'\\
&P\ar[u]^f\ar@{-->}[ul]^{\exists \ra f}
}\]
Now all comodules are finite dimensional, and because $P$ is projective within finite
 dimensional comodules, there exists a lifting $\ra{f}:P\to Z'\subseteq Z$ of $f$,
hence, a lifting of the original $f$.
\end{proof}

For clarity of the exposition we recall the 
definition of suspension,  desuspension and triangles in $\uM^H$.
For  this particular case
of comodules over a co-Frobenius Hopf algebra, the general definitions
can be more explicitly realized. Moreover, for $H=H_0\#\B$ as in Section
\ref{H0B}, concrete  and functorial constructions can be done
in $\m^H$. The reader familiar with Happel's results may go directly 
to Section 2.

\subsection{Suspension  and desuspension functors}

In \cite{Ko}, when $H$ is finite dimensional and $\Int$ is an integral in $H$ (not in $H^*$),
the author embeds an $H$-module $X$ via  $X\ot \Int\subset X \ot H$ and define $T(X)$ as
$(X\ot H)/(X\ot \Lambda)$. For us,  $\Int\in H^*$ and this definition makes no sense, but
(even without using the integral) one can always embed  an $H$-comodule $M$ 
into $M\ot H$ by means
of its structural map. The structure map $\rho$ is $H$-colinear provided we use the 
(co)free $H$-comodule structure on $M\ot H$ (and not the diagonal one).

\begin{defi}For a right $H$ comodule $M$ with structure $\rho:M\to M\ot H$,
 define
\[
T(M):=(M\ot H)/\rho(M)\]
\end{defi}
If $H$ is finite dimensional this definition also makes sense in $\m^H$. If
$H$ is co-Frobenius and  $0\neq M$ is  finite dimensional,
$(M\ot H)/\rho( M)$ is not finite dimensional, however,
there exists a finite dimensional 
injective $I(M)$ and a monomorphism $M\to I(M)$, so,
one can define $I(M)/M$ in $\ul \m^H$ and we know $T(M)\cong I(M)/M$ in $\uM^H$.
Moreover, for co-Frobenius Hopf algebras, one can give functorial
embeddings $M\to I(M)$ in $\M^H$ that works in $\m^H$ (see 
Corollary \ref{injfunctorial}).

\begin{rem}
If the notation $M\ot H$ is confusing because $H$ is Hopf and one also has the diagonal 
action, one may consider another injecting embedding:
\[
i_M:M\to M\ot H
\]
\[
m\mapsto m\ot 1\]
This map is clearly an embedding, and it is $H$-colinear if one uses the diagonal action on 
$M\ot H$. Both embeddings are ok because $M\ot H$ with diagonal action and
$M\ot H$ with structure coming only from $H$ are isomorphic (see
Lemma \ref{lemainj}).
\end{rem}

\

Similar (or dually) to  \cite{Ko} one can define desuspension. Consider
the map  $\Int'=\Int\circ S$\[
\Int':H\to k
\]
 Recall that $H$ co-Frobenius implies $S$ is bijective
(see for instance \cite{DNR}) and it is easy to prove that $\Int'=\Int\circ S$ is a
{\em  right} integral:
\[
\Int'(h_1)h_2=
\Int(S(h_1))h_2=
\Int(S(h_1))S^{-1}S(h_2)\]
\[=
S^{-1}\Big(\Int(S(h_1))S(h_2)\Big)=
S^{-1}\Big(\Int(S(h)_2))S(h)_1)\Big)=
S^{-1}\Big(\Int(Sh)1)\Big)=
\Int(Sh)1=
\Int'(h)1\]
We use $\Int'$ because $\Int'$ is right colinear:
\[
\xymatrix{
h\ar@{|->}[d]\ar@{|->}@/^3ex/[rrr]&H\ar[d]^{\rho=\Delta}\ar[r]^{\Int'}&k\ar[d]^{\rho}
&\Int'(h)\ar@{|->}[dr]\\
h_1\!\ot\!  h_2\ar@{|->}@/_3ex/[rrr]&H\!\ot\!  H\ar[r]^{\Int'\!\ot\! \id}&k\!\ot\!  H
&\Int'(h_1)\!\ot\!  h_2=1\!\ot\! \Int'(h_1) h_2=1\!\ot\!  \Int'(h)1\ar@{=}[r]& \Int'(h)\!\ot\!  1
}
\]
Hence, $\Ker(\Int')$ is a right $H$-comodule and we have, for any $M$,  a short
exact sequence
\[
0\to M\ot \Ker(\Int')\to M\ot H\to M\to 0\]
\begin{defi}The desuspension functor is
 $T'(M):=M\ot \Ker(\Int')\in\ul \M^H$
\end{defi}
\begin{rem}
When considering $\m^H$, we know every  finite dimensional comodule $M$
has a finite dimensional projective cover $P(M)\to M$, so we can consider
$T''(M):=\Ker(P(M)\to M)\in\ul \m^H$ and this is isomorphic to $T'(M)$ in the 
stable category. Also (see Corollary \ref{injfunctorial}), one can define $P(M)$
in $\M^H$ and in $\m^H$ in a functorial way.
\end{rem}

As an illustration of the need of stabilization for having a triangulated category 
one see that, for any $M$ a comodule, we have a short exact sequence in $\M^H$
\[
0\to M\to M\ot H\to TM\to 0
\]
In particular, considering $T'M$ instead of $M$,  there is  a short exact sequence
\[
0\to T'M\to T'M\ot H\to TT'M\to 0
\]
But  there is also a short exact sequence
\[
0\to T'M\to M\ot H\to M\to 0
\]
So "$M$ computes $TT'M$ using another injective embedding".  
Usually $TT'M\not\cong M$ in $\M^H$ but $M\cong TT'M$ in $\ul\M^H$.
Similar argument for $T'T$, hence these are mutually inverse functors in the
{\em stable} category, but not in $\M^H$.

\subsection{Triangles}
One of the axioms of triangulated categories is  that any map $f:X\to Y$ is a part
of a triangle $X\overset{f}{\to}Y\to Z\to TX\to$. Triangles are defined via
the mapping cone construction.
 For $f:X\to Y$,  $Co(f)$
is  defined in the following way:

Choose an injective embedding $i:X\to I(X)$ and define $Co(f)$ by the square
\[
\xymatrix{
X\ar[r]^f\ar[d]^i&Y\ar[d]\\
I(X)\ar[r]&Co(f):=I(X)\prod_ X Y
}\]
One can see that this definition does not depend -in the stable category- on the choice 
of the injective embedding $X\to I(X)$. Notice also
a well defined map $Co(f)\to T(X)$ given by the universal property of the push-out:
\[
\xymatrix{
X\ar[r]^f\ar[d]^i&Y\ar[d]\ar@/^2ex/[rdd]^0&\\
I(X)\ar[r]\ar@/_2ex/[rrd]_\pi&I(X)\prod_ X Y\ar@{-->}[rd]&\\
&&I(X)/X}\]

 Triangles $X\to Y\to Z\to TX$ in $\uM^H$ are (by definition)  all
sequences isomorphic (in $\uM^H$) to some sequence of the form
$X\overset{f}{\to} Y\to Co(f) \to T(X)$.
Next two Lemmas emphasize the strong relation between the exact structure of 
$\M^H$ (resp $\m^H$) and the triangulated structure of
$\ul\M^H$ (resp $\ul\m^H$)

\begin{lem}\label{lema1}
If
$ \xymatrix@-2ex{
0\ar[r]& X\ar[r]^u& Y\ar[r]^\pi& Z\ar[r]& 0}
$
is a short exact sequence in $\M^H$ then the sequence $X\to Y \to Z$
is isomorphic to $X\to Y\to Co(u)$ in the stable category.
\end{lem}

\begin{proof}
We assume $Z=Y/u(X)$.
Consider the diagram
\[
 \xymatrix@-1ex{
0\ar[r]& X\ar@{=}[d]\ar[r]^u& Y\ar[r]^\pi\ar@{=}[d]& Y/u(X)\ar[r]& 0\\
& X\ar[r]^u& Y\ar[r]^v& Co(u)&
}
\]
Let $X\to I(X)$ be an embedding into an injective object, for simplicity
we assume $X \subseteq I(X)$.
We define the map 
\[
I(X)\oplus Y\longrightarrow Z\]
\[
\hskip 1cm (e,y)\mapsto \pi(y)
\]
It has the property that, for any $x\in X$,
\[
(-x,u(x))\mapsto \pi(u(x))=0
\]
So, it induces a well defined map
\[
Co(u)=\frac{I(X)\oplus Y}{(x,0) \sim (0,u(x))}\longrightarrow Z\]
\[
\hskip 2cm \ra{(e,y)}\mapsto \pi(y)
\]
Now from the injectivity of $I(X)$ we know there exists a map fitting into the diagram
\[
 \xymatrix{
0\ar[r]& X\ar[d]\ar[r]^u& Y\ar[r]^\pi\ar@{-->}[ld]_ U& Y/u(X)\ar[r]& 0\\
&I(X)
}\]
So, define the map
\[
Y\to Co(u)
\]
\[
y\mapsto \ra{(U(y),y)}
\]
It has the property
\[
u(x)\mapsto \ra{(-U(u(x)),u(x))}=\ra{(-x,u(x))}=0
\]
so, it induces a well defined map
\[
Z=Y/u(X)\to Co(u)
\]
One composition is the identity:
\[
Z\to Co(u)\to Z
\]
\[
z=\pi(y)\mapsto \ra{(-U(y),y)}\mapsto \pi(y)=z
\]
The other composition is
\[
 Co(u)\to Z\to Co(u)
\]
\[
\ra{(e,y)}\mapsto \pi(y)\mapsto \ra{(-U(y),y)}
\]
so, the Kernel is
\[
\{\ra{(e,y)} : y\in u(X)\}\cong
\frac{I(X)\oplus u(X)}
{(x,0) \sim (0,u(x))}\cong I(X)
\]
that is an injective comodule, so, these morphisms
are mutually  inverses in $\uM^H$.
\end{proof}

The second Lemma is a useful one, maybe it is folklore but it is not 
usually written:
\begin{lem}\label{remsec}\label{lema2}
If $\xymatrix@-2ex{X\ar[r] &Y\ar[r] & Z\ar[r] &TX}$ is a triangle in
 the stable category
then there exists a short exact sequence
$0\to X'\to Y'\to Z'\to 0$ in $\M^H$ such that the sequence
$\xymatrix@-2ex{X\ar[r]&Y\ar[r] & Z}$  is isomorphic to
$\xymatrix@-2ex{X'\ar[r] &Y'\ar[r] & Z'}$ in 
the stable category.
\end{lem}
\begin{proof}
One of the axioms of triangulated categories says that
 $\xymatrix@-2ex{X\ar[r] &Y\ar[r] & Z\ar[r] &TX}$ is a triangle
if and only if
 $\xymatrix@-2ex{T^{-1}Z\ar[r]& X\ar[r] &Y\ar[r] & Z}$ is so. Hence, 
 $\xymatrix@-2ex{T^{-1}Z\ar[r]& X\ar[r] &Y\ar[r] & Z}$ is
 isomorphic to a distinguished triangle, that is, there is an isomorphism 
(in the stable category) of t-uples
\[
\xymatrix@-2ex{
T^{-1}Z\ar[r]\ar[d]_\cong& X\ar[r]\ar[d]_\cong &Y\ar[r]\ar[d]_\cong & Z\ar[d]_\cong\\
A\ar[r]^u&B\ar[r]&Co(u)\ar[r]&T(A)
}\]
In particular, there is a commutative diagram in the stable category
\[
\xymatrix@-2ex{
 X\ar[r]\ar[d]_\cong &Y\ar[r]\ar[d]_\cong & Z\ar[d]_\cong\\
B\ar[r]&Co(u)\ar[r]&T(A)
}\]
and clearly $0\to B\to Co(u)\to T(A)\to 0$  -or equivalently
\[
0\to B\to I(A)\prod_AB\to I(A)/A\to 0,
\] 
 is a short exact sequence in $\M^H$.
Notice that if $A$ and $B$ are finite dimensional,  one can find a finite dimensional
injective hull $I(A)$ and hence the short exact sequence also belongs to $\m^H$.
\end{proof}

\section{Integrals and coinvariants \label{exint}}
If $C$ is a coalgebra and $M$ a {\em  right} $C$-comodule then 
$M$ is a {\em  left} $C^*$ module via
\[
\phi\cdot m:=\phi(m_1)m_0\hskip 1cm (\phi\in C^*,\ m\in M)
\]
where, as usual, if $M$ is a right $H$-comodule, we denote $\rho:M\to M\ot H$ its
 structural map and we use  Seedler-type notation
$\rho(m)=m_0\ot m_1\in M\ot H$.
In particular, for $C=H$ and $\phi=\Int\in H^*$,
 being left integral means $\Lambda\cdot h = \Lambda(h)1$.
Moreover,  multiplication by $\Int$ in $M$ has the following standard 
and main property:
\[
\rho(\Int \cdot m)=\rho(\Int(m_1)m_0)
=\Int(m_2)m_0\ot m_1
=m_0\ot \Int(m_2)m_1
=m_0\ot \Int(m_1)1
= (\Int\cdot m)\ot 1
\]
That is, $\fbox{
$\Lambda\cdot M\subseteq M^{coH}$
}$.  We list some examples, keeping in mind the above formula.
\subsection*{Examples}
\begin{enumerate}
\item If  $H$ is co-semisimple 
(e.g. 
$H=\O(G)$ with  $G$ an affine reductive group)
then the inclusion $k\to H$ split as $H$-comodules. One 
can check that an $H$-colinear splitting is an integral.
In the cosemisimple case, the inclusion
$\Lambda\cdot M\subseteq M^{coH}$ is actually an equality (this will
be clear in Subsection \ref{int}).
Nevertheless, the integral may not be so explicitly described.
 An easy and explicit example is:
 \item  If
$G$ is a finite group and $H=k^G$, then $\Int=\sum_{g\in G}g\in k[G]\cong (k^G)^*$
is an integral. For any $f\in k^G$:
\[
\Int(f)=\sum_{g\in G}f(g)\]
Actually, every  {\em finite dimensional} Hopf algebra is (Frobenius and)  co-Frobenius.
Notice that $k^G$ is co-semisimple if and only if $k[G]$ is semisimple, if and
only if  the characteristic of the ground field does not divide the order of $G$.

\item Let 
$G$ be a group (possibly infinite, e.g. $G=\Z$) and  $H=k[G]$, define
\[
\Int(\sum_{g\in G}\lambda_g g):=
\lambda_{1_G} \]
A right $H$-comodule $M$ is the same as $G$-graded vector space 
$M=\oplus_{g\in G}M_g$.
The action of $\Int$ gives  the projection into $M_{1_G}$.

\item Tensor product of co-Frobenius algebras is co-Frobenius, the integral can be
 computed using tensor products of integrals. 

\item Let $H$ be a Hopf algebra and $H_0$ its coradical. Notice
 that $H_0$ does not need to be a Hopf subalgebra in general. Nevertheless, one
of the main results in \cite{ACE} is that $H$ is co-Frobenius {\em if and
only if}  the coradical filtration is finite. A particular case is illustrated in the following:

\item \label{exmix}
 Let $H_0$ be a cosemisimple Hopf algebra and let $V\in{}_{H_0}\YD^{H_0}$
be a finite dimensional Yetter-Drinfel'd module such that
its Nichols algebra 
$\B=\B(V)$ is finite dimensional. Then $H=H_0\#\B$ is co-Frobenius. 
The integral is essentially given
by the "volume form", or "Fermionic integration" in $\B$ (see Remark \ref{remP}).
\begin{enumerate}
\item The simplest example is: $H$ generated by $x$ and $g^{\pm 1}$ with 
relations $x^2=0$ and $gx=-xg$.  Comultiplication given by
\[
\Delta g=g\ot g\]
\[
\Delta x=x\ot g+1\ot x\]
The antipode is
\[
S(g)=g^{-1},\
S(x)=-xg^{-1}=g^{-1}x\]
We have  $H\cong k[\Z]\#k[x]/x^2$.
An element of $h$ may be uniquely written as 
\[
h=\sum_{n\in\Z}a_ng^n+
\sum_{n\in\Z}b_ng^nx
\hskip 2cm (a_n,b_n\in k)
\]
A left  integral is given by 
\[
\Int(h):=b_{0}
\]
{This particular example motivates  all definitions of this paper.}
The second simplest example of this kind is the following:
\item $H$ generated by $x,y$ and $g^{\pm 1}$ with relations $x^2=0=y^2$, $xy=-yx$,
$gx=-xg$, $gy=-yg$ and comultiplication given by
\[
\Delta g=g\ot g\]
\[
\Delta x=x\ot g+1\ot x\]
\[
\Delta y=y\ot g^{-1}+1\ot y\]
If we write an element $h\in H$ as
\[
h=
\sum_{n\in\Z}a_ng^n
+\sum_{n\in\Z}b_ng^nx
+\sum_{n\in\Z}c_ng^ny
+\sum_{n\in\Z}d_ng^nxy
\]
then a left integral is given by $\Int(h)=d_0$.
We will compute some invariants of the (stable) comodule category associated to 
this $H\cong k[\Z]\#\Lambda(x,y)$.

\end{enumerate}
\end{enumerate}

One of the main goals of this paper is to translate into Hopf-co-Frobenius language 
the notion of homology "$\Ker d/ \Im d$". The definition
is very natural:

\subsection{Hopf homology for algebras with a non-zero integral \label{int}}

\begin{defi}
Given a co-Frobenius Hopf algebra $H$ and $M\in\M^H$,
denote
\[
\fbox{\fbox{$\H_0^H(M):=\dfrac{M^{coH}}{\Int\cdot M}$}}
\]
For $n\in\N$, we define
\[
\H_{-n}^H(M):=
\H_0^H(T^nM)
\]
and
\[
\H_{n}^H(M):=
\H_0^H(T'{}^nM)
\]
\end{defi}

\begin{ex}
If $M=k$ and $H$ is  co-Frobenius with $\Int(1)= 0$, 
then $\Int \cdot k=0$, hence $\H_0^H(k)=k$ and
the functor $\H_ 0^H$ is non trivial.
\end{ex}

\begin{ex}
For $M=H$, $\Int\cdot H =k1_H=H^{co H}$
$\To$  $\H_ 0^H(H)=0$.
\end{ex}

\begin{ex} The condition ``$M^{coH}/\Int\cdot M=0$'' is stable under arbitrary 
direct sums and direct 
summands, so
$I^{coH}/\Int\cdot I=0$ for any injective module $I$.
\end{ex}
As a corollary,  
if $f:M\to N$ is an $H$-colinear map such that if factors through an injective:
\[
\xymatrix{
M\ar[rr]^f\ar[rd]&&N\\
&I\ar[ru]
}
\]
then the induced map
\[
\xymatrix{
\H_ 0^H(M)\ar[rr]^{\H_ 0^H(f)}\ar[rd]&&\H_ 0^H(N)\\
&\H_ 0^H(I)\ar[ru]\ar@{=}[r]&0
}
\]
is necessarily zero. So, the functor $\H_ 0^H(-)$ is actually defined
in the stable category 
\[
\H_ 0^H:\ul \M^H\to {}_kVect
\]
\begin{rem}
For all $n\in\Z$, the functors $\H_n^H(-)$ are defined in the stable category.
\end{rem}

\begin{coro}
If $H$ is co-semisimple then  every $H$-comodule is injective, hence
$\H_ 0^H(M)=0$ for all  comodule $M$. In other words, 
$\Int\cdot M=M^{coH}$ for all comodule $M$.
\end{coro}

\begin{rem}
From the point of view of invariant theory, 
$\Int\cdot M=M^{coH}$ is the most convenient situation, but from
the point of view of homological algebra,
$\H_ 0^H(M)\neq 0$ is  most interesting.
\end{rem}

\begin{lem}\label{propepi}
Let $H$ be a Hopf algebra with nonzero integral $\Int$ and denote $\ul \Hom^H$
the Hom space in the stable category of $H$ comodules, then there
 exists an {\bf epimorphism}
\[
\ul\Hom^H(k,M)\to \H_ 0^H(M)
\]
\end{lem}
\begin{proof}
Notice that 
\[
\Hom^H(k,M)\to M^{co H}
\]
\[
f\ \ \mapsto  \ \  f(1)
\]
is an isomorphism. We will show that this map fits into a commutative square
\[
\xymatrix{
\ar@{->>}[d]\ar[r]^{\cong}_{f\longmapsto f(1)} \Hom^H(k,M)&\ar@{->>}[d]M^{co H}\\
\ar@{-->}[r]  \ul\Hom^H(k,M) &M^{co H}/\Int\cdot M
}
\]
Assume $f:k\to M$ factors through an injective object
\[
\xymatrix{
k\ar[rr]^f\ar[rd]&&M\\
&I\ar[ru]
}
\]
then one may consider the unit map $k\overset{\eta}{\to} H$ and  the diagram
\[
\xymatrix{
k\ar[d]_\eta\ar[rr]^f\ar[rd]&&M\\
H\ar@{-->}[r]&I\ar[ru]
}
\]
Since $\eta$ is a monomorphism and $I$ is injective, one may find a
 dashed morphism making a commutative
 diagram, so, it is enough to consider the case $I=H$.
\[
\xymatrix{
k\ar[rr]^f\ar[rd]&&M\\
&H\ar[ru]_b
}
\]
Now if $x\in H$ is such that $\Int(x)=1$, then
\[
f(1)=b(1)=b(\Int(x)1)=b(\Int\cdot x)=\Int\cdot b(x)\in\Int \cdot M
\]
so $f(1)\in\Int\cdot M$. This proves that the induced map
\[
 \ul\Hom^H(k,M) 	\to M^{co H}/\Int\cdot M
\]
is both well-defined,  and clearly surjective.
\end{proof}


\begin{rem}
One may wonder if the epimorphism of the above Lemma is in fact an isomorphism. This will be the
case (see Theorem \ref{teoiso}). For finite dimensional Hopf algebras it is due to You Qi \cite{Q2}, where
he proves actually for finite dimensional Frobenius algebras, in particular for finite dimensional
Hopf algebras. It is not clear for the author how to adapt Qi's arguments
to our case, maybe one can find a simpler proof, but we provide a proof
with some homological machinery first.
\end{rem}

\begin{rem}
Lemmas \ref{lema1} and \ref{lema2}  gives an alternative proof that 
the composition of two consecutive
 morphisms in a triangle is zero (in the stable category), and so, {\em every} functor
defined in the stable category sends triangles to complexes.
For the particular case of $\H_ \bullet^H(-)$, without knowing that it is representable or not,
 we have the expected result:
\end{rem}
\begin{teo} \label{les}
If $X\to Y\to Z\to TX$ is a triangle in the stable category  then there 
is a long exact sequence of vector spaces
\[
\cdots \to  \H_{n+1}^H(Z)\to \H^H_n(X)
 \to \H^H_n(Y)\to \H^H_n(Z)
\to \H^H_{n-1}(X)\to\cdots\]
\end{teo}

\begin{proof}
We will prove that 
\[
 \H^H_0(X)\to 
 \H^H_0(Y)\to \H^H_0(Z)
\]
is exact in $\H_ 0^H(Y)$ when $0\to X\to Y\to Z\to 0$ is a short exact sequence. 
The general result 
follows from Lemma \ref{remsec} and the shifting axiom of triangles.
So assume $0\to X\to Y\to Z\to 0$ is a short exact sequence in $\M^H$, then
multiplication by the integral gives as a commutative diagram (of vector spaces)
with exact rows
\[
\xymatrix{
0\ar[r]& X\ar[r]\ar[d]^{\Int\cdot}& Y\ar[r]\ar[d]^{\Int\cdot}&Z\ar[d]^{\Int\cdot}\ar[r]& 0 \\
0\ar[r]& X^{coH}\ar[r]& Y^{coH}\ar[r]& Z^{coH}& \\
}\]
So, even forgetting that $X\to Y$ is injective, the snake Lemma gives in particular that
\[
\xymatrix{
X^{coH}/\Int\cdot X\ar[r]& Y^{coH}/\Int\cdot Y\ar[r]& Z^{coH}/\Int\cdot Z
}\]
is exact.
\end{proof}
Now, the above Theorem together with Lemma \ref{propepi}
gives the following:
\begin{teo}\label{teoiso}
Let $H$ be a Hopf algebra with nonzero integral $\Int$ and denote $\ul \Hom^H$
the Hom space in the stable category of $H$ comodules, then
the natural map
\[
\ul\Hom^H(k,M)\to \H_ 0^H(M)
\]
is an isomorphism.
\end{teo}
\begin{proof}First 
recall the following version of the "5"-lemma: given a commutative diagram 
with exact rows:
\[\xymatrix{
A\ar[r]\ar[d]^a&B\ar[r]\ar[d]^b&C\ar[r]\ar[d]^c&D\ar[d] ^d\\
A'\ar[r]&B'\ar[r]&C'\ar[r]&D'\\
}\]
if $b$ and $d$ are monomorphisms and $a$ is an epimorphism, then $c$
is a monomorphism.

Consider $\S$ the class of $H$-comodules $S$ such that the map
$\ul\Hom^H(k,S)\to \H_ 0^H(S)$ is an isomorphism. Because
short exact sequences in $\M^H$ gives both long exact sequences for
$\H^H_n(-)$ and $\ul\Hom^H(k,T^{-n}(-))$, given a short exact sequence of comodules
\[
0\to S_1\to M\to S_2\to 0\]
where $S_1$ nd $S_2$ are in $\S$, then 
we have a diagram
\[\xymatrix{
\ul\Hom^H(k,T^{-1}S_2)\ar[d]^a\ar[r]&\ul\Hom^H(k,S_1)\ar[r]\ar[d]^b
&\ul\Hom^H(k,M)\ar[r]\ar[d]^c&\ul\Hom^H(k,S_2)\ar[d]^d\\
\H_0(T^{ -1}S_2)\ar[r]&\H_0(S_1)\ar[r]&\H_0(M)\ar[r]&\H_0(S_2)\\
}\]
Every vertical  map is an epimorphism (Lemma \ref{propepi}) and  both $b$ and $d$
are monomorphism because they are isomorphisms ($S_i\in\S$),
so $c$ is monomorphism, hence, an isomorphism. 

We conclude that the theorem is true for any finite dimensional comodule $M$,
provided it is true on simple comodules.

If $S=k$ and $k$ is not injective then $\Int\cdot k=0$ and
$\H_0^H(k)=k\cong \ul\Hom^ H(k,k)$. (If $k$ is injective, the theorem is noninteresting, but still true).

If $S$ is simple and $S \not\cong k$ then $S^ {co H}=0$, so trivially 
$\H_ 0^H(S)=0=\ul\Hom(k,S)$.

Now let  $M$ be a possibly  infinite dimensional comodule and $f:k\to M$ such that 
$f(1)=\Int\cdot m$ for some $m\ni M$.  Consider $M'\subset M$ a finite dimensional
 subcomodule containing $f(1)$ and $m$. Then,
the class of $f(1)$ in $\H_ 0^H(M')$ is zero. But because $M'$ is finite dimensional
 we know $\H^H_0(M')=\ul\Hom^H(k,M')$ and so
there exists a factorization
\[
\xymatrix{
k\ar[r]^f\ar[d]&M'	\ar@{^(->}[r]&M\\
I\ar[ru]}
\]
with $I$ injective. So,  $f$ is zero in $\ul\Hom^H(k,M)$.
\end{proof}

\subsection{Multiplicative structure}

Because $H$ is Hopf, the categories $\M^H$ and $\m^H$ are tensor categories,
and the tensor structure descends to the stable category, as one can see after
these standard facts:

\begin{lem}
\begin{enumerate}
\item If $C$ is a coalgebra and $V$ a vector space, the right $C$ comodule
$V\ot C$ with structure map $\rho=\id_V\ot\Delta$ is an injective comodule.
\item Every injective comodule is a direct summand of one as above.
  The category of comodules has enough injectives.
\end{enumerate}
\end{lem}
\begin{proof}
\begin{enumerate}
\item It follows from the adjunction formula
\[
\Hom^C(M,V\ot C)\cong \Hom_k(M,V)
\]
\[
f\mapsto (\id_V\ot\epsilon)\circ f\]
and that every vector space is an injective
object in $k$-Vect.

\item If $M$ is a comodule, the structure morphism
\[
\rho_M:M\to M\ot C\]
gives an embedding into an injective object: $C$-colinearity is by coassociativity and
injectivity is because of counitarity. If $M=I$ is injective,
then the monomorphism $:\rho_I:I\to I\ot C$ splits, hence,
$I$ is a direct summand of $V\ot C$ where $V$ is the underlying vector space of $I$.

\end{enumerate}

\end{proof}

\begin{lem}\label{lemainj}
Let $H$ be a Hopf algebra, $M\in\M^H$. Denote
 $V_M$ the underlying vector space of $M$.
\begin{enumerate}
\item $M\ot H$ (with diagonal coaction)
is isomorphic to $V_M\ot H$ (with 
 $\rho=\id_{V_M}\ot\Delta_H$).
\item Also $H\ot M\cong V_M\ot H$

\item  If $I$ is injective 
then $M\ot I$ and $I\ot M$ are both injectives.
\end{enumerate}
\end{lem}
\begin{proof}
\begin{enumerate}

\item We only exhibit the maps:
\[
M\ot H\to V_M\ot H\]
\[
m\ot h\mapsto m_0\ot m_1h
\]
with inverse
\[
m\ot h\mapsto m_0\ot S(m_1) h\]
The composition is
\[
m\ot h\mapsto m_0 \ot S(m_1) m_2h
= m_0 \ot \epsilon(m_1)h
=m \ot h
\]
The other composition is similar.
The surprising part is that these maps are $H$-colinear. 
For instance:
\[
\xymatrix@-2.5ex{
m\ot h\ar@{|->}@/^3ex/[rrr]\ar@{|->}[d]
&H\ot H\ar[r]\ar[d]^{\rho_{diag}} &V_H\ot H\ar[d]^{\id_V\ot \Delta}
& m_0\ot m_1h\ar@{|->}[d]\\
m_0\!\ot \! h_1\! \ot\!  m_1h_2\ar@{|->}@/_3ex/[rrr]
&H\!\ot\! H\!\ot\! H\ar[r]&V_H\!\ot\! H\!\ot\! H
&m_0\!\ot\! m_1h_1\!\ot\! m_2h_2
=m_0\!\ot\! (m_1h)_1\!\ot\! (m_1h)_2\\
}\]

\item The maps are similar: consider
\[
H\ot M\to V_M\ot H\]
\[
h\ot m \mapsto m_0\ot m_1h
\]
with inverse
\[
m\ot h\mapsto  S(m_1)h \ot m_0\]
The composition is
\[
h\ot m\mapsto m_0 \ot  m_1h
= S(m_1)m_2h\ot m_0
= \epsilon(m_1)h\ot m_0
= h\ot m
\]
The other composition is similar.
The colinearity follows the same lines.

\item If $I$ is injective then it is isomorphic to a direct summand of $V\ot H$ for some
 vector space $V$ (e.g. $V=V_I$), and so
$M\ot I$ is isomorphic to a direct summand of
\[
M\ot (V\ot H)\cong V\ot (M\ot H)\cong (V\ot V_M)\ot H
\]
and $I\ot M$ is a direct summand of
\[
( V\ot H)\ot M\cong V\ot (H\ot M)\cong (V\ot V_M)\ot H
\]
In any case, a direct summand of a comodule of the form $W\ot H$ for some vector 
space $W$.
\end{enumerate}
\end{proof}

There are several corollaries:
\begin{coro}
The tensor product is well defined in the stable category. In particular,
$K_0(\ul\m^H)$ is an associative ring.
\end{coro}

Let $E:=E(k)$ be the injective hull of $k$ in $\M^H$.
 It is well-known that $H$ is co-Frobenius 
if and only if $E(k)$ is finite dimensional (see Theorem 2.1 in  \cite{AC}).
Also, for co-Frobenius Hopf algebras, there exists a finite dimensional projective
comodule $P=P(k)$ with a surjective map $P\to k$.

\begin{coro}\label{injfunctorial}
Define $E(M):=M\ot E$.
The map $M\to E(M)$  ($m\mapsto m\ot 1$) is a functorial injective embedding,
if $M\in\m^H$ then $E(M)\in\m^H$ as well. Also, $P(M):=M\ot P$ gives a functorial
projective surjection $P(M)\to M$, if $P\in\m^H$ then $P(M)\in\m^H$ as well.
\end{coro}
\begin{proof}
The injective part is clear. Let us prove the existence of a surjective map $P\to k$ with $P$ finite dimensional:

Since $\Int':H\to k$ is surjective, there exists $h_0\in H$ such that $\Int'(h_0)=1$, and
there exists a finite dimensional subcomodule $M_0\subset H$ containing $h_0$. In particular,
$\Int'(M_0)\neq 0$. Because $H$ is co-Frobenius, there exists a finite dimensional
injective hull of $M_0$, let's call it $I(M_0)$. Looking at the diagram
\[
\xymatrix{
0\ar[r]&\ar[d]M\ar[r]&I(M)\ar@{-->}[ld]^\exists\\
&H
}\]
Because $H$ is injective there exist the dashed arrow. Because
$I(M_0)$ is the injective hull and $H$ is injective, the map $I(M)\to H$ is injective
and $I(M_0)$ is a direct summand of $H$. Eventually changing $M_0$ by $I(M_0)$ we get
a finite dimensional direct summand of $H$ such that the restriction of $\Int'$
is non-zero, hence, a surjection $P\to k$ with $P$ projective and finite dimensional.
\end{proof}

\begin{coro}For any $M$, $N$ in $\M^H$, there are isomorphisms
in the stable category
\[
T(M\ot N)\cong TM\ot TN\cong
TM\ot N\cong M\ot TN\]
and similarly for $T'$.
Hence, $\ul\M^H$ and $\um^ H$ are tensor triangulated categories.
\end{coro}

\begin{proof}
Let $i:M\to I(M)$ and $j:N\to I(N)$ be  embeddings into injective comodules, then
$I(M)\ot I(N)$ is injective and one can compute $T(M\ot N)$ via
\[
0\to M\ot N\overset{i\ot j}{\to} I(M)\ot I(N)\to T(M\ot N)\to 0
\]
But 
$I(M)\ot N$ and $M\ot I(M)$ are injectives too, and
we  have the following short exact   sequences with injective objects in the middle:
\[
\xymatrix@-1px{
0\ar[r]&M\ot N\ar[r]^{i\ot \id}&I(M)\ot N\ar[d]^{\id\ot j}\ar[r]&TM\ot N\ar[r]\ar@{-->}[d]&0\\
0\ar[r]&M\ot N\ar[r]^{i\ot j}\ar@{=}[u]\ar@{=}[d]&I(M)\ot I(N)\ar[r]&T(M\ot N)\ar[r]&0\\
0\ar[r]&M\ot N\ar[r]^{\id\ot j}&M\ot I(N)\ar[r]\ar[u]^{i\ot \id}&M\ot T N\ar[r]\ar@{-->}[u]&0\\}
\]
\end{proof}
Notice that the morphisms are not canonical in the category of comodules, but 
they are canonically determined in the stable category

\begin{coro}\label{coroT}
For any $m,n\in\Z$  there is an isomorphism in the stable category
\[
T^nM\ot T^nN\cong T^{n+m}(M\ot N)
\]
\end{coro}

\subsection*{K\"unneth map}

Let $M$ and $N$ be two comodules. It is clear that
$
M^H\ot N^H\subseteq (M\ot N)^H
$
and also one can easily check that
\[\Int\cdot(M\ot N^{coH})=\Int\cdot M\ot N^{coH}\]
and
\[\Int\cdot(M^{coH}\ot N)= M^{co H}\ot \Int\cdot N\]
So, there is a canonical map
\[
\H_0^HM\ot \H_0^HN=
\frac{M^{co H}}{\Int\cdot M}
\ot
\frac{N^{co H}}{\Int\cdot N}
\cong
\frac{M^{coH}\ot N^{coH}}{\Int\cdot M\ot N^{coH}+M^{coH}\ot \Int \cdot N}
\longrightarrow
\H_0^H(M\ot N)
\]
Moreover, using Corollary \ref{coroT} on can define maps 
\[
\xymatrix@-2ex{
\H^H_p(M)\ot \H^H_q(N)\ar@{=}[d]\ar@/^1ex/[rdrd]&&\\
\H^0_H(T'{}^pM)\ot \H^0_H(T'{}^qN)\ar[d]&&\\
\H^0_H(T'{}^pM\ot T'{}^qN)\ar[r]_\cong & 
\H^0_H(T'{}^{p+q}(M\ot N))\ar@{=}[r]&\H^H_{p+q}(M\ot N)}
\]
(If a number is negative, we use the convention $(T')^{-n}=T^n$.)
In this way, one can assembly all those maps and get, for any fixed $n$, 
 a map that we call "K\"unneth map"
\[
\bigoplus_{p+q=n}\H^H_p(M)\ot \H^H_q(N)\to \H^H_{n}(M\ot N)
\]
For  $M=N=k$, from concrete computations (see Corollary \ref{HC(k)}) we know this 
map cannot be an  isomorphism in general.
It would be interesting to know their general properties. In any  case,
$\H_\bullet^H(k)=\bigoplus_{n\in \Z}\H_n^H(k)$ is a graded algebra.

\section{Small injective  embeddings
for $\m^{H_0\#\B}$   \label{H0B} }

During this section we assume 
\begin{enumerate}
\item $H_0$ is a co-semisimple Hopf algebra,
\item  $V\in {}_{H_0}\YD^{H_0}$ is such that $\B(V)$, the Nichols algebra associated 
to the braided vector space $V$, is finite dimensional. 
\end{enumerate}
Let us recall briefly the conditions above and set notations and conventions.
 First, ${}_{H_0}\YD^{H_0}$ is the category
whose objects are
left $H_0$-modules and right $H_0$-comodules with the compatibility
\[
h_1m_0\ot h_2m_1=(h_2m)_0\ot (h_2m)_1)h_1
\]
where $h_1\ot h_2=\Delta h$, $h\in H_0$ and $m\in M$, $\rho(m)=m_0\ot m_1\in M\ot H_0$.
Morphisms are $H_0$-linear and colinear maps.
For any Hopf algebra $A$, the category
${}_{A}\YD^{A}$ is braided with
\[c_{V,W}:V\ot W\to W\ot V\]
\[
v\ot w\mapsto w_0 \ot w_1\cdot  v
\]

Recall that if $V$ is a braided vector space (e.g. $V\in \YD_{H_0}^{H_0}$)
then both $TV$ (the tensor algebra) and $T^cV$ (the tensor coalgebra)
are {\em braided} Hopf algebras. $TV$ has free product and braided-shuffle coproduct, 
while $T^cV$ has deconcatenation coproduct and braided-shuffle product.
The Nichols algebra  $ \B(V)$ is, by definition, 
 the image of  the unique (bi)algebra map $TV\to T^cV$
that is the identity on $V$: 
\[
\xymatrix{
TV\ar@{->>}[rd]\ar[rr]&& T^cV\\
&\B(V)\ar
@{^(->}
[ru]
&
}
\]
It happens to be, degree by degree,  the image of the quantum symmetrizer map
 associated to the braiding. We refer to Andruskievitch's notes \cite{Aleyva}
for a gentle introduction and full discussion on Nichols algebras. 
The reader may keep in mind the easy example $\B(V)=\Lambda V$ when the braiding is 
-flip. The braided bialgebra $\B(V)$ is actually a braided Hopf algebra, and the
bicross product $H_0\#\B(V)$ is a usual Hopf algebra. Since there is a lot of structures 
around $\B(V)$ we recall them:

\begin{itemize}
\item $\B$ is a coalgebra, we denote $\Delta (b)=b_1\ot b_2$,
\item $\B\in\M^{H_0}$, we denote the structure $\rho (b)=b_0\ot b_1$,
\item $H_0\#\B$ is a coalgebra, the comultiplication is given by the following diagram
(recall the underlying vector space of $H_0\#\B$ is $H_0\ot \B$):
\[
\xymatrix{
H_0\ar[d]^{\Delta_{H_0}}&\ot&\B \ar[d]^{\Delta_\B }\\
 H_0\ot H_0\ar@<-2ex>[d]_{\id_{H_0}}\ar[drr] |\hole&\ot& \B \ot \B \ar@<2ex>[d]^{\id_\B }
\ar[lld]_{C_{H_0,\B }}\\
H_0\ot \B  &\ot& H_0\ot \B \\
}\]
In Sweedler-type notation:
\[
\Delta( h\#b)=h_1 \#(b_1)_0\ot (b_1)_1h_2\# b_2
\]
\item In particular $\Delta( 1\#b)=1 \#(b_1)_0\ot (b_1)_1\# b_2$.
Denoting $H:=H_0\#\B$, we have that $\B\cong 1\#\B$ is a right $H$-subcomodule of $H$.
With this structure we consider $\B$ as an object in $\M^H$. To emphasize the difference
with $\rho:B\to \B\ot H_0$ we call it $\rho_H$.
\end{itemize}

\begin{ex}\label{ex31} Let $x\in V\subset \B$, $\Delta x=x\ot 1+1\ot x$. Assume
$\rho(x)=x\ot g\in \B\ot H_0$. In order to compute $\rho_H(x)$ we proceed
as follows:
\[
\rho_H(x)\leftrightarrow \Delta_H(1\# x)
=
1 \# (x_1)_0\ot (x_1)_1 \# x_2
\]
\[
=1 \#(x)_0\ot (x)_1\# 1
+1 \#(1)_0\ot (1)_1\# x
\]
\[
=1 \#x\ot g\# 1
+1 \#1\ot 1\# x
\leftrightarrow 
x\ot g
+1 \ot  x
\]\end{ex}

The main fact of this section is the following:
\begin{prop}\label{propinj}
$\B\in\M^H$ is an injective object.
\end{prop}

\begin{proof}
Since $H_0$ is cosemisimple, the inclusion $k\to H_0$ splits as $H_0$-comodule.
Choose a splitting $\Int'_0:H_0\to k$. This is actually right integral for $H_0$, that is, 
it satisfies
\[
\Int'_0(h_1)h_2=\Int_0'(h)1
\]
and additionally $\Int'_0(1)=1$.

Now we define a splitting $H\to \B$ of the inclusion $\B\cong 1\#\B\subset H_0\#\B=H$
via
\[
h\# b\mapsto \Int'_0(h)b
\]
We need to see that it is $H$-colinear. Recall the $H$-structure in $\B$ is
given by the identification $\B\cong1\#\B\subset H$, so
\[
\rho_H(b)=b_{1_0}\ot b_{1_1}\# b_2
\]
We check $H$-colinearity:
\[
\xymatrix@-1ex{
h\#b\ar@{|->}@/^3ex/[rrr]\ar@{|->}[d]&
H_0\#\B\ar[r]^\pi\ar[d]^{\Delta_{H_0\#\B}}&\B\ar[d]^{\rho_H}&
\Int'_0(h)b\ar@{|->}[d]\\
h_1  \# b_{1_0} \!\ot\!  b_{1_1}h_2\#b_2\ar@{|->}@/_2ex/[rrd]&
H_0\#\B \!\ot\!  H_0\#\B \ar[r]&\B \ot   (H_0\#\B) &
\Int'_0(h)b_{1_0}\ot b_{1_1}\# b_2\ar@{=}[d]\\
&&
\Int'_0 (\!  h_1 \! ) b_{1_0} \!\ot\!  b_{1_1}  h_2\#b_2\ar@{=}[r]&
 b_{1_0} \!\ot\!  b_{1_1}\Int'_0 (\!  h_1 \! )h_2\#b_2
\\
}\]

\end{proof}

\begin{rem}
The proof is independent of the fact of  $\B$ being finite dimensional,
but we are interested in the case $\dim\B<\infty$ so that $H_0\#\B$ is co-Frobenius.
\end{rem}

As a corollary we have
\begin{coro}\label{coroinj}
For any $M\in\M^H$, the map $i_M:M\to M\ot\B$ defined by 
\[
m\mapsto m\ot 1
\]
is a functorial injective embedding. In particular, if $\B$ is finite dimensional
then  $M\to M\ot \B$ ($m\mapsto m\ot 1$) is a finite dimensional 
embedding working in $\m^H$. From the short exact sequence
\[
0\to M\to M\ot\B\to (M\ot \B)/M\to 0
\]
we have $TM\cong (M\ot\B)/M$. Recall $\B$ is graded (with the tensor grading)
and $\B_{top}$ (its maximal degree) has dimension 1. The Kernel of $\Int'|_\B$
is $\B_{<top}=\oplus_{i=0}^{top-1}\B_i$. From
$\pi:\B\to \B/\B_{<top}$ we have the short exact sequence
\[
0\to M\ot \B_{<top}\to M\ot \B\to M\ot (\B/\B_{<top})\to 0\]
hence $T'(M\ot (\B/\B_{<top}))\cong M\ot \B_{<top}$.
\end{coro}
\begin{rem}\label{remP}
$\B/\B_{<top}$ is {\em not } isomorphic to $k$ in general, but it is 1-dimensional.
So, in order to compute $T'M$ one should 
{\em "twist $M\ot \B_{<top}$ by the inverse of the quantum determinant":}

If $\b$ is a generator of the 1-dimensional vector space $\B_{top}$ then $k\b$
is not in general an $H$-subcomodule of $\B$, but $\B/\B_{<top}=k\ra\b$ is
so, hence
\[\rho_H(\ra \b)=\ra\b\ot D\]
for a unique group-like element $D\in H$, that we call "quantum determinant".
From the surjective map
\[
\B\overset{\pi}{\to} \B/\B_{<top}\cong kD\]
 we get a  surjective map into the trivial comodule $k$:
\[
\B\ot kD^{-1}\to k
\]
If we call $P:=\B\ot kD^{-1}$, it is a projective $H$-comodule that surjects into $k$
and from it one has functorial projective surjections for any comodule $M$:
\[
P(M):=M\ot P\to M
\]
and functorial $T'$, since from:
\[
0\to M\ot \B_{<top}\ot kD^{-1}\to
 M\ot P\to M\to 0\]
we get
$T'(M):=M\ot B_{<top}\ot kD^{-1}$ is a functor in $\M^H$ (resp. in $\m^H$
if $M$ is finite dimensional) that gives the desuspension functor in $\uM^H$
(resp. in $\um^H$).
\end{rem}

Before going into $K_0$ rings, we look at  some examples.

\section{First examples}

\subsection{The example $k[\Z]\# k[x]/x^2$ \label{THEexample}}
Let $H$ be the $k$-algebra generated by $x$ and $g^{\pm 1}$ with relations
\[
gx=-xg, \ x^2=0
\]
It is a Hopf algebra if one defines the 
comultiplication by
\[
\Delta g= g\ot g, \ \Delta x=x\ot g+1\ot x
\]
(to be compared with Example \ref{ex31}).
That is, $g$ is group-like and $x$ is 1-$g$-primitive. Notice that $k[x]/x^2$
is not a Hopf algebra in the usual sense
 (unless characteristic=2), but it is a super Hopf algebra. Nevertheless, 
$H=k[\Z]\#(k[x]/x^2)$ is a Hopf algebra in the usual sense.
Maybe all computations in this example are folklore, but for clearness 
we include them.

For an element
\[
\om=\sum_{n\in\Z}a_ng^n+\sum_{n\in\Z}b_ng^nx
\]
define
\[
\fbox{$\Int(\om)
:=b_{0}
$}
\]
The main fact about the category  $\M^H$, noticed by Bodo Pareigis \cite{P}, is


\begin{center}
\fbox{\em A right  $H$-comodule $M$ is the same as a d.g. structure on $M$}
\end{center}

\

Notice that evaluation at $x=0$ gives a map $H\to k[\Z]$, so any $H$-comodule is 
a $k[\Z]$-comodule (i.e.  a $\Z$-graded object), but the presence of $x$ keep track of a square-zero differential.
We just write the correspondence:
if $M=(\oplus_{n\in\Z}M_n,\partial)$ with $\partial(M_n)\subseteq M_{n-1}$ and
$\partial^2=0$ then, for $m\in M_n$, the right comodule structure 
is
\[
\fbox{$
\rho(m)=m\ot g^{n}+\partial(m)\ot xg^{n-1}
$}\]
and every right $H$-comodule is of this form.


It is a pleasant exercise to check that the standard differential on the tensor products
of complexes with the usual Koszul $\partial\ot 1\pm 1\ot \partial$ agree with
the standard  $H$-comodule structure on the tensor product of $H$-comodules.

Notice that
\[
\Delta 1=1\ot 1\]
\[
\Delta x=x\ot g+1\ot x\]
means that $k[x]/x^2$ is a right $H$-subcomodule of $H$. As d.g. vector space is the
 complex
\[
\cdots \to 0\to kx\underset {x\mapsto 1}{\overset{\partial}{\to}} k\to 0\to\cdots\]
where $|x|=1$, $ |1|=0$.

\subsubsection*{Smaller injective embeddings for
$H=k[\Z]\#k[x]/x^2$}

In this case we have $\B=k\oplus kx$, considered as $H$-comodule
via
\[
\rho(1)=1\ot 1
\]
\[
\rho(x)=x\ot g+1\ot x\]
The general argument developed in the previous section gives
us that, for any $M\in\M^H$, the  map $M\to I(M):=M\ot \B =M\ot k[x]/x^2$
given by $i(m)=m\ot 1$ is an embedding of $M$ into an injective object.
In particular $I(M)/\rho(M)=M\ot x$ and we have proven the following:
\begin{coro}
In the stable category of $\M^H$ for $H=k[\Z]\#(k[x]/x^2)$,
\[
TM=I(M)/M\cong M\ot kx\cong M[1]
\]
\end{coro}
We leave as an   exercise the following:
\begin{coro}
Identifying  d.g.$Vect_k$ and $\M^H$, 
the comodule
$M\ot k[x]/x^2$  identifies with the mapping cone of the identity of $M$.
Moreover, the "stable $H$-comodule mapping cone" of a colinear map 
identifies with the classical mapping cone of a map between complexes.
\end{coro}

\subsubsection*{Homology }

For a d.g. vector space $M=(\oplus_n M_n,\partial)$
 viewed as $k[\Z]\# k[x]/x^2$-comodule, the
 coinvariants are
\[
M^{coH}=\{m : \rho(m)=m\ot 1_H=m\ot g^0\}
\]
But "$\rho(m)=m\ot g^0$ " means that $m\in M_0$ and $\partial m=0$,
so
\[
M^{coH}=\Ker(\partial:M_0\to M_{-1})\]
On the other side, the action 
of the integral on an element $m$ gives
\[
\Int\cdot m=(\id\ot\Int)\rho(m)
=(\id\ot\Int)\rho(\sum_ nm_n)\]\[
=(\id\ot\Int) (\sum_ n(m_n\ot g^n+\partial(m_n)\ot g^{n-1}x))
=\partial  (m_{1})\in M_0\]
That is
\[
\Int\cdot
(\oplus_n M_n,\partial)=\Im(\partial:M_ {1}\to M_0)
\subseteq \Ker(\partial:M_0\to M_{-1})\subseteq M_0
\]
Hence
\[
\fbox{
\fbox{
$\H_0^{k[\Z]\#k[x]/x^2}(M)=M^{coH}/\Int\cdot M=H_0(M,\partial)$
}}\]

\subsection{The example $k[\Z]\#k[x]/x^N$ and $N$-complexes}

Fix $N\in \N$, $N\geq 2$.
Let $H$ be the algebra generated by $g^{\pm 1}$ and $x$ with the relations
\[
x^N= 0, \ gx=\xi_Nxg
\]
where $\xi_N$ is an $N$-primitive root of unity. This algebra is Hopf
 with comultiplication
\[
\Delta g= g\ot g, \Delta x=x\ot g+1\ot x
\]
To have an $H$-comodule is the same as a $\Z$-graded vector space together with
a degree -1 map $\partial$ satisfying $\partial^N=0$. The tensor structure 
for $(M_\bullet,\partial_M)\ot (M'_\bullet,\partial_R)$ 
is given by
the usual total grading in $M\ot M'$, and the differential on homogeneous elements
is
\[
\partial(m\ot m')=\partial(m)\ot m' + \xi_N^{|m|}m\ot\partial (m')
\]
For an homogeneous element $m\in M$ of degree $n$, the coaction
is given by
\[
\rho(m)=m\ot g^n+\partial(m)\ot xg^{n-1}
+\frac{1}{[2]_\xi}\partial^2(m)\ot x^2g^{n-2}+
\cdots 
+\frac{1}{[N-1]_\xi!}\partial^{N-1}(m)\ot x^{N-1}g^{n-N+1}
\]
\[
=\sum_{i=1}^{N-1}
\frac{1}{[i]_\xi!}\partial^{i}(m)\ot x^{i}g^{n-i}
\]
where as usual 
$[0]_\xi!=[1]_\xi!=1$, 
$[n]_\xi=1+\xi+\cdots+\xi^{n-1}$ and $[n+1]!_\xi=[n+1]_\xi\cdot [n]!_\xi$.

If $(M_\bullet,\partial)$ is an $N$-complex, there are several ways to associate an
 "homology" in degree $n$. For each $0<i<N$, since $0=\partial^N=\partial^i\partial^{N-i}$,
one may consider $\Ker(\partial^i) / \Im(\partial^{N-i})$. The general machinery
of co-Frobenius algebras and stable categories, however, choose one particular $i$.
Since
$M^{coH}=\Ker(\partial)\cap M_0$ and $\Int\cdot M=M_0\cap \Im(\partial^{N-1})$
we have
\[
\H_ 0^H(M)=\frac{\{m\in M_0 : \partial(m)=0\}}
{\partial^{N-1}(M_{-N+1})}
\]
The other (homological) degrees are not the $\H_0$ of the  degree-shiftings
of the $N$-complex. As an illustration we compute $\H_1(M)$ in terms of the 
$N$-complex data:
\begin{prop} $H_1(M)\cong
\dfrac{ \Ker(\partial^{N-1})\cap M_{N-1}}{\Im( \partial:M_{N}\to M_{N-1})}
$
\end{prop}

\begin{rem} From Corollary \ref{coroinj} and the isomorphism 
$\B/\B_{<top}\cong kg^{N-1}$ (notice
$\B/\B_{<top}$ is generated by the class of $x^{N-1}$ and
$\rho_H(x^{N-1})=x^{N-1}\ot g^{N-1}$+lower degree terms)
 it follows that
\[
T'(M)\cong M\ot kg^{-N+1}\ot \B_{<top}
\]
\end{rem}

\begin{proof}
Recall $\B=\bigoplus_{i=0}^{N-1}kx^i$, the structure is given by
\[
|x^i|=i,\ \partial(x^i)=[i]x^{i-1}
\]
\[
\H_ 1(M)=\H_0(T'M)
=\H_0(M\ot kg^{-N+1}\ot (k\oplus kx\oplus kx^2\oplus\cdots \oplus kx^{N-2}))\]
The degree zero part of $T'M$, if $M=\oplus_n M_n$, is
\[
T'(M)_0=\bigoplus_{i=0}^{N-2} M_{N-i-1}\ot g^{-N+1}\ot x^i
\]
A typical element is of the form
\[
\sum_{i=0}^{N-2}m_{i}\ot g^{N-1}\ot x^i
\]
where $|m_i|=-N-i+1$.
The differential has the form
\[
\partial\Big(\sum_{i=0}^{N-2}m_{i}\ot g^{N-1}\ot x^i\Big)=
\sum_{i=0}^{N-2}\partial(m_{i})\ot g^{N-1}\ot x^i+
\sum_{i=0}^{N-2}\xi_N^{-N-i+1+N-1}m_{i}\ot g^{N-1}\ot \partial(x^i)
\]
\[
=
\sum_{i=0}^{N-2}\partial(m_{i})\ot g^{N-1}\ot x^i+
\sum_{i=1}^{N-2}\xi_N^{-i}m_{i}\ot g^{N-1}\ot [i]x^{i-1}
\]
\[
=
\partial(m_{N-2})\ot g^{N-1}\ot x^{N-2}+
\sum_{i=0}^{N-3}
\Big(
\partial(m_{i})+
\xi_N^{-i-1}[i+1]m_{i+1}\Big)\ot g^{N-1}\ot x^{i}
\]
This expression is  equal to zero if and only if
\[
\partial(m_{N-2})=0
\hskip 1cm \hbox{ and } \hskip 1cm 
m_{i}=\frac{-\xi^i_ N}{[i]}\partial(m_{i-1})
\hskip 1cm (i=N-3, N-4,\dots,1)
\]
From the second set of equalities we see that the only parameter is $m_0$,
because $m_i$ is, up to scalar, $\partial^i(m_0)$. The equation
$\partial(m_{N-2})=0$  means $\partial^{N-1}(m_0)=0$.
We conclude
\[
(T'M)^{co H}\cong \Ker(\partial^{N-1})\cap M_{N-1}\]
We leave to the reader to check that, under this bijection,  
$\Int\cdot (M\ot g^{N-1}\ot \B_{<top})$ corresponds to
$\partial(M_{N})$.
\end{proof}

\subsection{The example $k[\Z]\#\Lambda(x,y)$ and Mixed complexes}
Denote
\[
\Lambda(x,y):=k\{x,y\}/(x^2,y^2,xy+yx)
\]
It is not a Hopf algebra in the usual sense, but it is a Hopf algebra in the (signed) 
graded sense. The algebra
\[
k[\Z]\#\Lambda(x,y)=k\{g^{\pm 1},x,y\}/
(
gx=-xg,gy=-yg,0=x^2=y^2=xy+yx)\]
is a Hopf algebra with comultiplication
\[
\Delta g= g\ot g,\
\Delta x=x \ot g+1\ot x,\
\Delta y= y\ot g^{-1}+1\ot y\]
Notice that 
$x$ produce a differential of degree -1, while $y$ produce a differential of degree +1.
This Hopf algebra $H$ is isomorphic to $H_0\#\B(V)$ where $H_0=k[\Z]$ and  
$V=kx\oplus ky\in{}_ {H_0}\YD^{H_0}$.
Writing $\Z$ multiplicatively
$\Z\cong\{ g^n: n\in\Z\}$, the action is given by
\[
gv=-v,\ \forall v\in V\]
and the coaction is determined by
\[
\rho x = x\ot g,\ \rho y = y\ot g^{-1}
\]

\begin{lem}
$\M^H$ identifies with 
objects $(M,d,B)$ where $M$ is a $\Z$-graded vector space, $d$ and $B$ are square zero 
differentials with $|d|=-1$, $|B|=1$, and $dB+Bd=0$. In other words,
$\M^H$ are {\bf mixed complexes}.
\end{lem}

The proof is straightforward, we only indicate the correspondence: for a mixed
complex $(M,d,B)$, the corresponding right comodule structure
\[
\rho:M\to M\ot H
\]
for an homogeneous $m$,  is
given by
\[
\rho(m)=m\ot g^{|m|} 
+d(m)\ot xg^{|m|-1}  + B(m)\ot yg^{|m|+1}+dB(m)\ot yxg^{|m|}
\]
It is clear that $M^{co H}=M_ 0\cap\Ker d\cap \Ker B$. Also, an easy 
computation  shows (see Example \ref{exmix}(b) of Section 
\ref{exint} for the expression of the integral)
\[
\Int\cdot M= d(B(M_0))=B(d(M_0))\subseteq M^{coH}
\]

So, 
\[
\H^0(M)=
\frac{M_ 0\cap\Ker d\cap \Ker B}
{d(B(M_0))}\]

\begin{rem}
In this (stable) category, the suspension functor is {\em not} the shifting degree in
 general. However, we have the following lemma:
\end{rem}

\begin{lem}$k[x]/x^2$ and $k[y]/y^2$ are $H$-subcomodules of $H$. 
 $\B=\Lambda(x,y)\cong k[x]/x^2\ot k[y]/y^2$ 
as objects in $\M^H$.
For  $M\in \M^H$ denote  $M(x):=M\ot k[x]/x^2$ and
$M(y):=M\ot k[y]/y^2$. The following assertions follows:
\begin{itemize}
\item $M(x)(y)\cong M(y)(x)\cong 
M\ot \Lambda(x,y)$ is an injective object in $\M^H$.
\item $T(M(x))\cong M(x)[-1]$, 
\item $T(M(y))\cong M(y)[1]$
\item $\H_\bullet(M(x))=H_{-\bullet}(M,B)$
\item $\H_\bullet(M(y))=H_{\bullet}(M,d)$
\end{itemize}
\end{lem}
\begin{proof}
The first item follows from the obvious isomorphism
\[
k[x]/x^2\ot k[y]/y^2\cong \Lambda(x,y)
\]
Observe that $kx=(k[x]/x^2)/k=kx\cong k[1]$ and
$ky=(k[y]/y^2)/k=ky\cong k[-1]$. Now
from the short exact sequence
\[
0\to M(x)\to M(x)\ot k[y]/y^2\to M(x)\ot y\to 0\]
we get
\[
0\to M(x)\to M\ot \Lambda(x,y)\to M(x)[-1]\to 0\]
Since $M\ot \Lambda(x,y)$ is injective, we conclude
$T(M(x))\cong M(x)[-1]$. Similarly for $M(y)$.

In order to  compute cohomology we consider first
\[
M(x)^{co H}=M(x)_0\cap\Ker d\cap \Ker B=M(x)_0^{ d,B}\]
We visualize it using the following diagram
\[
\xymatrix{
M_{-1}\ar@/^/[r]^d &M_0\ar@/^/[r]^d\ar@/^/[l]^B &M_1\ar@/^/[r]^d\ar@/^/[l]^B&\cdots\\
\cdots&\ar@{=>}[u]^dM_0\ot x\ar@/^/[r]^d\ar@/^/[l]^B&\ar@{=>}[u]^dM_1\ot x\ar@/^/[r]^d\ar@/^/[l]^B&M_2\ot x\ar@/^/[l]^B&\\
}\]
So, 
\[
M(x)^{d}=\{(m_0, m_1\ot x) : d(m_0)+m_1=0,\ d(m_1)=0\}
\]
\[=
\{(m_0, -dm_0\ot x) : m_0\in M_0\}\cong M_0
\]
\[
M(x)^{d,B}=\{(m_0, -dm_0|x) : Bm_0=0,\ B(-dm_0)=0\}\]
but $B(-dm_0)=dBm_0=0$, so
$M(x)^{d,B} \cong M_0^B$.

We also must compute $B(d(M(x)_0))$:
\[
Bd(m_0,m_1\ot x)
=B(dm_0+m_1,dm_1\ot x)
=(Bdm_0+Bm_1,Bdm_1\ot x)\]\[
=(B(m_1+dm_0),-d(B(m_1+dm_0))\ot x)
=(B\wt m,-d(B\wt m)\ot x)
\]
So, under the isomorphism $M(x)^{d,B}\cong M_0^ B$,
the subspace
$\Int\cdot M(x)=Bd(M(x)_0)$ corresponds to $B(M_1)\subset M_0$.
We conclude $\H^0(M(x))\cong H_ 0(M,B)$. 

Now from the second item we get
\[
\H_n(M(x))=\H_0(T^{-n}(M(x)))=\H_0(M(x)[n])=H_0(M(x)[n])=H_{-n}(M,B)
\]
The parts with $M(y)$ instead of $M(x)$ is completely analogous.
\end{proof}

\begin{coro} For any mixed complex $(M,d,B)$ there are
long exact sequences
\[
\cdots\to  \H_\bullet(M)\to H_{\bullet}(M,d)\to
\H_\bullet(M[-1])\to \H_{\bullet-1}(M)\to\cdots\]
and
\[
\cdots\to \H_\bullet(M)\to H_{-\bullet}(M,B)\to
\H_\bullet(M[1])\to \H_{\bullet-1}(M)\to\cdots\]

\end{coro}
\begin{proof}
We consider the short exact sequences in $\M^H$:
\[
0 \to M\to M\ot k[y]/y^2\to M\ot y\to 0
\]
and
\[
0 \to M\to M\ot k[x]/x^2\to M\ot x\to 0
\]
Recall $M\ot y\cong M[-1]$
and  $M\ot x\cong M[1]$.
These short exact sequences in
$\M^H$ gives triangles in
the stable category;  their log exact sequences together with
 the previous Lemma gives the result.
\end{proof}

\begin{coro}
 $H_n(M,d)=0$ $ \forall n \To \H_n(M)=\H_0(M[-n])$;\\
 $H_n(M,B)=0$ $ \forall n \To \H_n(M)=\H_0(M[n])$.
\end{coro}

Another corollary is the following computation:

\begin{coro}
\label{HC(k)} Considering $k$ as trivial mixed complex concentrated in degree zero,
\[
\H_{\bullet}(k)=
\left\{
\begin{array}{cl}
k&\bullet  = 0 \\
k&\bullet = -1\\
0&\hbox{ otherwise}
\end{array}
\right.\]
\end{coro}
\begin{proof}
Specializing the long exact sequence 
\[
\cdots\to  \H_\bullet(M)\to H_{\bullet}(M,d)\to
\H_\bullet(M[1])\to \H_{\bullet-1}(M)\to H_{\bullet-1}(M,d)\to\cdots\]
at  $M=k[p]$ and $\bullet=q+1$
gives
\[
\cdots\to 
 H_{q+1}(k[p],d)\to\H_{q+1}(k[p+1])\to 
 \H_{q}(k[p])\to H_{q}(k[p],d)\to\cdots
\]
If $p\neq q,q+1$ we have
\[
\H_{q+1}(k[p+1])\cong
 \H_{q}(k[p])
\]
Inductively, for $n\neq 0,1$
\[
\H_n(k)
=\H_n(k[0])\cong 
\H_{n-1}(k[-1])\cong \cdots \cong  \H_{0}(k[-n])=0
\]
because $k[n]$ do not have 0-degree component if $n\neq 0$. It remains to compute
$\H_{0}(k)$. and $\H_1(k)$.

Clearly $\H_0(k)=k$. For $\H_{1}$, since
$\B/\B_{<top}\cong k$ (notice $xy\in\B_{top}$ has degree zero),
the formula for
$T'$ is
\[
T'(k)=k\oplus kx\oplus ky\]
We have
$(T'k)_0=k=(T'k)^{co H}$
and $dB=Bd=0$ in $T'k$, so
$\H_{1}(k)=\H_0(T'k)=k$.
\end{proof}

\begin{rem}
Notice the asymmetry in the gradings, $\H_1(k)=k$ but $\H_{-1}(k)=0$, as we can see from 
the general argument above, or compute directly:
\[T(k)=\B/k=kx\oplus ky\oplus kxy\]
The degree zero component is $kxy$, but $d(xy)\neq 0$ (also $B(xy)\neq 0$),
so $T(k)^{co H}=0$ and
\[
\H_{-1}(k)=\H_0(T(k))=\frac{T(k)^{co H}}{\Int\cdot T(k)}=0
\]

\end{rem}

\section{(De)Categorification: computation of  $K_0$}

\subsection{$K_0$ of exact and triangulated categories}
The main result of this section is Theorem
\ref{teok}, that gives a general presentation of $K_0(\um^H)$.
 We recall the main constructions:

If $\A$ is an {\em exact} category such that the isomorphism classes of objects is a set, then
$K_0(\A)$ is defined as the free abelian group on the set of isomorphism
classes of objects module the relations
\[
[X]+[Z]=[Y]\]
whenever there is a short exact sequence
\[
0\to X\to Y\to Z \to 0
\]
We remark that $[X]+[Y]=[X\oplus Y]$ and, for $n\in\N$, $n[X]=[X^n]$, so, every
element in $K_0(\A)$ can be written in the form $[X]-[Y]$ for some
objects $X,Y$ in $\A$.
For {\em triangulated} categories, $K_0$ is defined similarly, taking
the free abelian group
on isomorphism classes of objects modulo the relations
\[
[X]+[Z]=[Y]\]
whenever there is a triangle
\[
 X\to Y\to Z \to TX
\]

By $K_0(\m^H)$ we understand the K-theory of the category of finite
dimensional $H$-comodules, that is an exact category with usual short 
exact sequences. We denote
$K_0(\um^H) $ the K-theory of  the stable category $\um^H$ as triangulated
category.

Almost by definition, if  $\I$ denotes the full subcategory of injective objects in
 $\m^H$, there is a short exact 
exact sequence of categories
\[
0\to \I\to \m^H\to \um^H\to 0
\]
One could expect a general result in K-theory 
concluding a long exact sequence
ending with
\[
K_0(\I)\to K_0( \m^H)\to K_0( \um)^H\to 0\]
This is actually the case for short exact sequences of {\em exact} categories  where
the left hand side is also a Serre subcategory.  Recall a
Serre (sub)category is closed under quotients, subobjects and 
extensions. In our case, 
$\m^H$ is an exact category but $\I$ is not a Serre subcategory in general.
Also, $\um^H$ is not in general an exact category, it is triangulated, but the other
two are not triangulated.  

 For exact sequences
of Waldhausen
categories there is also a long exact sequence in $K$-theory, 
both $\I$ and $\m^H$ are Waldhausen categories, but it is not clear
that
$\um^H$ is so, in any case one should prove it. Instead, one can prove directly
the following:

\begin{teo}\label{teok}The natural functors $\I\to\m^H$ and  $\m^H\to\um^H$ 
induce  an exact sequence
\[
K_0(\I)\to K_0( \m^H)\to K_0( \um^H)\to 0\]
In particular, the ring
$ K_0( \um^H)$ can be presented as the quotient of
$ K_0( \m^H)$ by the ideal generated by  injective objects.
\end{teo}
\begin{proof}
The functor $\m^H\to \um^H$ is the identity on objects, so the
induced map
$K_0(\m^H) \to K_0(\um^H)$ is surjective.
By definition of $\um^H$, the composition $\I\to\m^H\to\um^H$ is zero, so 
the composition $K_0(\I)\to K_0( \m^H)\to K_0( \um)^H$ is zero as well.
Let us denote
\[
K_0(\m^H)/K_ 0(\I):=\frac{K_0(\m^H)}
{
\Im\big(
K_0(\I)\to K_0( \m^H)\big)}
\]
We have a surjective ring homomorphism
\[
K_0(\m^H)/K_ 0(\I)\to K_0(\um^H)
\]
To see injectivity of this map we argue as follows:
Assume $\om=[M]-[N]$ in $K_0(\m^H)$ goes to zero in
 $K_0(\um^H)$. 
From the short exact sequence
\[
0\to N\to I(N)\to TN\to 0\]
in $\m^H$ with $I(N)$ injective, $[TN]=[I(N)]-[N]$ in
 $K_0( \m^H)$, hence
 $[TN]=-[N]$ in $K_0( \m^H)/K_0(\I)$.
 So, we have
\[
\om=[M]-[N]=[M\oplus TN]\ Mod\ \I
\]
Eventually changing $\om=
[M]-[N]$ by $\om':= [M\oplus TN]$, we can assume
that, modulo $\I$, the element $\om$ is equal to $[M]$
for some object $M$.
Now if $[M]$ is zero in
$K_0(\um^H)$, then  there exists integers $n_i$ and triangles
in the stable category
\[
 X_i\to Y_i\to Z_i \to TX_i
\]
such that
\[
[M]=\sum_im_i([X_i]+[Z_i]-[Y_i])
\]
But, using that direct sum of triangles is a triangle, for the positive $m_i$'s
we get
\[
\sum_im_i([X_i]+[Z_i]-[Y_i])=
[\oplus_{m_i>0}X_i^{m_i}]+[\oplus_{m_i>0}Z_i^{m_i}]-[
\oplus_{m_i>0}Y_i^{m_i}]
\]
and similarly for the negative $m_i$'s.
From this, we may assume that there are two triangles
$ X_i\to Y_i\to Z_i \to TX_i$, $i=1,2$ such that
\[
[M]=([X_1]+[Z_1]-[Y_1])
-([X_2]+[Z_2]-[Y_2])
\]
But  because $X\overset{\id}{\to}X\to 0\to TX$ is a triangle, then so is
 $X\to 0\to TX\to TX\to$, hence $[TX]=-[X]$ in $K_0(\um^H)$, and
for $X\to Y\to Z\to TX$ a triangle, we have $TX\to TY\to TZ\to T^2X$ is also a triangle and
\[
-([X]+[Z]-[Y])=[TX]+[TZ]-[TY]
\]
So, we can conclude that there exists a single triangle $X\to Y\to Z\to $ such that
\[
[M]=[X]+[Z]-[Y]
\]
But we know (Lemma \ref{lema2})
that any triangle in the stable category $X\to Y\to Z\to TX$ is isomorphic,
in the stable category, to a short exact
 sequence 
\[
0\to X'\to Y'\to Z'\to 0
\]
Recall also that  $X\cong X'$ in $\um^H$ if and only if 
 there exist injectives $I$ and $J$ such that
$X\oplus I\cong X'\oplus J$ in $\m^H$.
But Modulo $\I$, clearly $[X]=[X]+[I]=[X\oplus I]=[X'\oplus J ]=
[X']+[J]=[X']$.  So, we finally get that
\[
[M]=[X']+[Z']-[Y']\ Mod \ \I
\]
Hence, $[M]$ in $K_0(\m^H)$ is zero Mod $K_0(\I)$.
\end{proof}

\begin{rem}
It could be interesting to know if this is the last part of a long exact sequence for higher 
K-groups.
\end{rem}

\subsection{$K_0$ and the coradical}

Let $H$ be a Hopf algebra and $H_0$ its coradical.
Since $H_0$ is a subcoalgebra, every  $H_0$-comodule is an $H$-comodule.
 Consider the category $\A=\m^H$
and $\BB=\m^{H_0}$; $\BB$  is a  non-empty full subcategory closed under taking
subobjects, quotient objects, and finite products in $\A$.
 Also $\BB$ is an abelian category and the inclusion functor $\BB\to \A$
is exact, so Quillen's theorem gives.

\begin{teo}(\cite{Q}, Theorem 4. (Devissage)) Let
$\BB$ and $\A$ be as above.
 Suppose that every object $M$ of $\A$ has a finite filtration
$0 = M_0\subseteq M_1\subseteq \cdots\subseteq M_n=M$ such that $M_j/M_{j-1}$
is in $\BB$ for each $j$. Then the inclusion $\BB\to \A$ induces an isomorphism
$K_\bullet(\BB)\cong K_\bullet(\A)$
\end{teo}

If $M\in\M^H$ is a nonzero comodule,
then its socle $\soc(M)$ is a nonzero subcomodule that is actually an $H_0$-comodule
(see Exercise 3.1.2. of \cite{DNR}, page 117,  its solution on page 140).
If in addition $M$ is finite dimensional, considering $M/\soc(M)$ and induction in 
the dimension of $M$ one can easily define a finite filtration
\[
0 =\soc(M)\subseteq M_2\subseteq \cdots\subseteq M_n=M\]
such that $M_j/M_{j-1}=
\soc(M_j/M_{j-1})$, hence
$M_j/M_{j-1}\in\m^{H_0}$.  Now Quillen's theorem implies the following:

\begin{coro}\label{devissage}
As abelian groups, 
$K_0(\m^H)\cong K_0(\m^{H_0})$. If $H_0$ is a Hopf subalgebra then this isomorphism is also a ring isomorphism.
\end{coro}

\subsection{Smash products \label{smash}}

Let $H=H_0\#\B$ where $H_0$ is cosemisimple and $\B$ a finite dimensional 
{\em braided} Hopf algebra in $_{H_0}\YD^{H_0}$. 
 For an element
$M\in \M^H$, denote
$\gr M$ the associated graded with respect to the "socle filtration".
Recall that the assignment 
$[M]\mapsto [\gr M]$ implements  the isomorphism
$K_0(\m^H)\cong K_0(\m^{H_0})$.
If  $\{S_i:i\in I\}$ denote the set of (isomorphism classes of) simple objects in $\M^{H_0}$,
then , for 
$M\in \m^H$,
\[
\gr M\cong \oplus_{i\in I}S_i^{m_i}
\]
for uniques (and finite non zero) multiplicity integers $m_i=m_i(M)$.
We define
\[
[M]_{H_0}:=\sum_i m_i[S_i]=[\gr M] \hskip 1cm \in  K_0(\m^{H_0})=\bigoplus _{i\in I}\Z [S_i]\]
In particular $\B$ is a finite dimensional $H$-comodule, so it makes sense
\[
[\B]_{H_0}\in K_0(\m^{H_0})
\]
In the  case $H_0=k[G]$ with $G$ a group, the  isomorphism classes of 
simple comodules can be  parametrized by $\{kg\}_{g\in G}$ 
 and $kg\ot kh\cong kgh$,  so we identify $K_0(\m^{k[G]})\cong \Z[G]$. 
The  main result of this section is the following:
\begin{teo}\label{teoK0}Let $G$ be a group and $H_0=k[G]$. Assume
$H=H_0\#\B$, with finite dimensional $\B$.
The assignment
$[M]\mapsto [M]_{H_0}$ induces  an isomorphism of rings
\[
K_0(\um^H)\cong \Z[G]/( [\B]_{H_0} )
\]
\end{teo}

\begin{proof}
From Theorem \ref{teok} it follows that
$
K_0(\um^H)\cong 
K_0( \m^H)/K_0(\I)
$.
But from Corollary \ref{devissage} we know
\[
K_0(\m^H)\cong K_0(\m^{H_0})\]
\[
M\mapsto [ \gr M]
\]
For $H_0=k[G]$ we also know $K_0(\m^{k[G]})\cong  \Z[G]$. 
We need to identify $\K_0(\I)$ inside
$K_0(\m^H)\cong \Z[G]$.

Recall that $\B$ is injective and $\B^{co H}=k$. Now let $I$ be a finite dimensional
 indecomposable injective $H$-comodule. Because $I$ is indecomposable and injective, 
$\soc(I)$ is an  indecomposable $H_0$-comodule (here, injectivity of $I$
is essential), hence simple and
\[
\soc (I)\cong kg
\]
for some $g\in G$. Clearly $\wt I:=I\ot kg^{-1}$ is an 
injective indecomposable $H$-comodule and $(\wt I)^{co H}=k$.
Since $\wt I$ is injective, there exist a dashed morphism in the diagram:
\[
\xymatrix{
\ar@{_(->}[d] k\ar@{=}[r]&\soc \wt I\ar@{^(->}[r]&\wt I\\
\B\ar@{-->}
[rru]&\\
}\]
This map restricted to the socle is injective, so the map is injective and we have
$\dim\B\leq\dim \wt I$. But $\B$ is injective, so the same argument
in the opposite direction gives $\dim \wt \I\leq\dim\B$ and so $B\cong \wt I$.
In other words,
\[
I\cong\B\ot kg
\]
for some $g\in G$. We can conclude that if $I$ is finite dimensional
injective (non necessarily indecomposable) $H$-comodule, then there
exists integers $\{m_g: g\in G\}$ with
\[
I\cong\bigoplus_{g\in G}m_g\B\ot kg
\]
That is, $\Im\!(K_0(\I)\to K_0(\m^H))$ is the ideal
generated by $[\B]$.
\end{proof}

\subsection{Examples}
The first two examples are well-known:
\begin{enumerate}
\item  $H=k[\Z]\#k[x]/x^2$ :
\[
K_0(\um^H)\cong K_0(k[\Z])/\langle k[x]/x^2\rangle 
\cong \Z[z^{\pm 1}]/(1+z)\cong \Z
\]
\item (Khovanov)  $H=k[\Z]\#k[x]/x^N$:
\[
K_0(\um^H)\cong K_0(k[\Z])/\langle k[x]/x^N\rangle 
\cong \Z[z^{\pm 1}]/(1+z+\cdots z^{N-1})
\]
If $N$ is is a prime $p$ then
$K_0(\um^H)\cong \Z[\xi_p]$.

\item  $H=k[\Z]\#\Lambda(x,y)$ where $|x|=1$, $|x=-1$, then
\[
\Lambda(x,y)=k\oplus kx\oplus ky\oplus kxy\]
hence
$
[\gr \Lambda(x,y)]=1+z+z^{-1}+1=z^{-1}+2+z=z^{-1}(1+z)^2
$
and so
\[
K_0(\um^H)\cong \Z[z^{\pm 1}]/(1+z)^2
= \Z[z]/(1+z)^2\cong \Z[t]/t^2
\]
\item If $N_1,\dots,N_k\in\N$, for $1\leq i<j\leq k$, a list of
nonzero scalars  $q_{ij}\in k^\times$ is given, then define 
$H$ as the algebra generated by $x_1,\dots,x_k,g_1^{\pm 1},\dots,g_k^{\pm 1}$
with relations
\[
g_ig_j=g_jg_i \hskip 1.5cm (\forall i,j)\]
\[
x_ix_j=q_{ij}x_jx_i \hskip 1cm ( i<j)\]
\[
x_ig_j=q_{ij}g_jx_i \hskip 1cm (i<j)
\]
\[
 g_ix_j=q_{ij}x_jg_i\hskip 1cm (i<j)
\]
\[
x_i^{N_i}=0\hskip 3cm\]
It is a Hopf algebra 
with comultiplication given by
\[
\Delta g_i=g_i\ot g_i\]
\[
\Delta x_i=x_i\ot g_i+1\ot x_i\]
Then $H$ is a Hopf algebra of the form $ k[\Z^k]\#\B$.
The algebra $\B$ has monomial basis
$\{x_1^{n_1}\cdots x_k^{n_k}, 0\leq n_i<N_i\}$, so, 
writing $\Z[\Z^k]=\Z[z_1^{\pm 1},\dots, z_k^{\pm 1}]$, 
\[
[ \B]_{H_0}=\prod_{i=1}^k(1+z_i+z_i^2+\cdots +z_i^{N_i-1})
\]
Hence,
\[
K_0(\um^H)\cong \Z[z_1^{\pm 1},\cdots ,z_k^{\pm 1}]/
\prod_{i=1}^k(1+z_i+\cdots z_i^{N_i-1})
\]
\end{enumerate}

\begin{rem}
It would be interesting to compute $K_0(H_0\#\B)$ for some non pointed
cosemisimple $H_0$, for instance, $H_0=\O(G)$ with $G$ non abelian reductive
affine group.
\end{rem}

\section{$H$-comodule algebras and the category ${}_A\M^H$}
An $H$-comodule algebra $A$ is a $k$-algebra $A$ together with 
an  $H$-comodule structure such that the multiplication map

\[
A\ot A\to A\]
and the unit
\[
k\to A
\]
are $H$-colinear.
Usual examples are:
\begin{itemize}
\item $H=k[G]$: comodule algebra = $G$-graded algebra.
\item $G$ finite, $H=k^G$: comodule algebra = algebra with a  $G$-action
 by ring homomorphisms.
\item $G$ affine group, 
$H=\O(G)$: comodule algebra = algebra with a rational $G$-action.
\item $H=U\g$, comodule algebra =algebra with a $\g$-action acting by derivations.

\end{itemize}

For our purpose, the motivating  example is 
$H=k[\Z]\#k[x]/x^2$. In this case, an $H$-comodule algebra = d.g. algebra.

Also, if $H$ is any Hopf algebra and $A$ is any algebra, then $A$ viewed as trivial
 $H$-comodule is an $H$-comodule algebra.

The main fact for our interest  is the following:
\[
M\in {}_A\M^H, \ V\in\M^H \To M\ot V\in {}_A\M^H
\]
where $A$-module structure in $M\ot V$ is the one coming from $M$
and the $H$-comodule structure is the diagonal one. 
Moreover, if $M$ is finitely generated as $A$-module and $V$ is finite dimensional,
then $A\ot V$ is finitely generated as $A$-module. In this way, the subcategory
of  ${}_A\M^H$ consisting on $A$-finitely 
 generated modules, denoted by
${}_A\m^H$ ,  
 is naturally  a module over the category $\m^H$.
Following \cite{Ko}, we consider the restriction functor
\[
{}_A\M^H\to\M^H
\]
and define $M\in {}_A\M^H\to\M^H$ to be {\em acyclic} (or $H$-acyclic to emphasize
the role of $H$) if
$M$ is injective as $H$-comodule. In other words, if $M\cong 0$ in $\uM^H$.
A map $f:M\to N$ is called {\em quasi-isomorphimsm} (qis) if $f$ becomes
an isomorphism in $\uM^H$. 
Denote $\I_A$ the class of objects in ${}_A\M^H$ that are injective as $H$-comodules.
\begin{ex}\label{exI}
Let $M\in {}_A\M^H$ be an arbitrary object and $I\in\M^H$ an injective
 $H$-comodule. In virtue of Lemma \ref{lemainj}, 
$M\ot I\in \I_A$.
\end{ex}

If $M,N\in{}_A\M^H$, denote
$\I_A(M,N)$ the set of maps that factors through an object in $\I_A$.
The stable category  - or the $H$-derived category-  , denoted by 
${}_A\uM^H$ and also by $\D_H(A)$,
 is defined as the category with same objects as
${}_A\M^H$ but morphism 
\[
\Hom_{\D_H(A)}(M,N):=\frac{\Hom_A^H(M,N)}{\I_A(M,N)}
\]
The subcategory of $\D_H(A)$ whose objects are in ${}_A\m^H$ (i.e.
are finitely generated as $A$-modules) is denoted by $\D^c_H(A)$.

Recall that if $E=E(k)$ is the injective hull of $k$, $E$ is a finite
dimensional injective $H$-comodule (because $H$ is co-Frobenius), and for 
any $M\in{}_A\M^H$, then
\[
M\to M\ot E
 \]
is an embedding of $M$ into an acyclic object in ${}_ A\M^H$. 
If $P:=P(k)$ is a (finite dimensional) projective cover of $k$, then
\[
M\ot P\to M
 \]
is an epimorphism from an $H$-acyclic object in ${}_A\M^H$ into $M$.
If $M$ is finitely generated as $A$-module, then so is $M\ot E$ and $M\ot P$.
The definition of $TM$, of $T'M$ and of  the mapping cone of objects and maps
in ${}_ A\M^H$ (resp. in ${}_ A\m^H$)
actually gives  objects in ${}_ A\M^H$
(resp. in ${}_ A\m^H$). One can easily see that all constructions and proof's
of Happel's Theorem 2.6  in \cite{Ha}, when starting with objects in
in ${}_ A\M^H$
(resp. in ${}_ A\m^H$) always stay 
in ${}_ A\M^H$
(resp. in ${}_ A\m^H$). So
$\D_H(A)$ and $\D^c_H(A)$ are triangulated categories, and by Example
\ref{exI}, they are modules over $\uM^H$ and $\um^H$ respectively.

\begin{rem}
$K_0(\D^c_H(A))$ is a module over the ring $K_0(\um^H)$.
\end{rem}

\begin{ex}
If $A=k$ then $\D_H(k)=\uM^H$ and $\D^c_H(k)=\um^H$.
\end{ex}

\begin{ex} if $H=k[\Z]\#k[x]/x^2$ and $A$ is an ordinary algebra viewed
as trivial $H$-comodule algebra then 
$\D_H(A)=\D(A)$, the (unbounded) derived category of $A$. 
\end{ex}

\begin{ex}
If $A$ is a semisimple Hopf algebra and $H$ is a co-Frobenius Hopf algebra, we view $A$ 
as trivial $H$-comodule algebra, then
\[
{}_A\M^H\cong \M^{A^*\ot H}
\]
Since $A$ is semisimple,  $A^*$ is co-semisimple and $A^*\ot H$ is co-Frobenius.
In this case we have $\D_H(A)=\um^{(A^*\ot H)}$. Also
if $H=k[G]\#\B$ as in Section \ref{smash} then
$K_0(\D_H(A))=K_0(\um^{A^*\ot k[G]\#\B})$ is a quotient of
$K_0(\m^{A^*\ot k[G]})$. Assuming $k$ algebraically closed,
every simple corepresentation of the tensor product
$A^*\ot k[G]$ 
is given by the tensor product of a simple $A^*$-comodule and a simple $k[G]$-comodule, hence
$K_0(\m^{A^*\ot k[G]})= K_0(\m^{A^*}) \ot_\Z \Z[G]=
K_0(A) \ot_\Z \Z[G]$. 
\end{ex}


\subsection*{Enriched Hom}

If $M,N\in{}_A\M^H$, there are several Hom spaces that one can consider.
 We begin with the discussion for d.g. $A$-modules:

If $M$ and $N$ are d.g. A-modules, then one may consider
\begin{itemize}
\item Chain maps: $\Hom_A^H(M,N)$= maps preserving degree and commuting with 
the differential.
\item Chain maps up to homotopy: $\Hom_A^H(M,N)/\sim$, where $f\sim g$ if $f-g=dh+hd$
for some degree +1 $A$-linear map $h$.
\item The HOM complex: $\HOM_A(M,N)=\oplus_{n\in\Z}\HOM_A(M,N)_ n$
where $\HOM_A(M,N)_n$= $A$-linear maps of degree $n$. 
If $A$ is concentrated in degree zero (i.e. $A$ is a trivial $k[\Z]\#k[x]/x^2$-comodule)
then 
$\HOM_A(M,N)_n=\prod_{q\in\Z}\Hom_A(M_q,N_{q+n})$
\item Morphisms in the derived category: $\Hom_{\D_H(A)}(M,N)$.
\end{itemize}

In general $\HOM_A(M,N)$ 
is different from $\Hom_A(M,N)$.
 Assume for simplicity $A$ 
is an ordinary alegbra (i.e. d.g. algebra concentrated in degree zero), 
 if $M$ and $N$ have infinite 
nonzero degrees, then
\[
\Hom_A(M,N)=
\Hom_A(\oplus_p M_p,\oplus_qN_q)
\not \cong \bigoplus_n \Big(\prod_q \Hom_A(M_q,N_{q+n})\Big)\]
For instance, if $M=\oplus _n A[n]$ and $N=A$, then
\[
\Hom_A(\oplus_nA[n],A)\not\cong
\oplus_n \Hom_A(A[n],A)\]
Nevertheles, the set of chain maps agree with $B^0(\HOM_A(M,N))$ and
the set of chain maps up to homotopy is the same as $H_0(\HOM_A(M,N))$.

For general co-Frobenius Hopf alegbras (i.e. not necesariily finite dimensional ones)
one has the same "problems" but also analogous solutions. 
First of all, if $H$ is a (not finite dimensional) Hopf algebra, $A$ an $H$-comodule algebra
and $M,N\in{}_A\M^H$, then $\Hom_A(M,N)$ is not an $H$-comodule in general. For instance,
if $A=k=N$ and $M=H$, then $H^*$ is a not rational $H^*$-module, so 
it is not an $H$-comodule. In this way, if one consider
\[
\Hom_A(M,N)
\]
it is not expectable to have an object in $\M^H$. 

It is not clear to the author how to get an object in $\M^H$ analogous to $\HOM_A$
(maybe $\HOM_A(M,N):=\underset{\to}{\lim}_{ \mu} \Hom_A(M_\mu,N)$,  where $M_\mu$
 runs over all $A$-finitely generated subobjects?).
 To have an object $\HOM_A(M,N)\in\M^H$ would provide
the notion of map up to homotopy just by taking $H_0$.
Nevertheles, we have the following
\begin{prop}
 $\Hom_A(M,N)$ is a (left) $H^*$-module and the definition of $\H_0^H$
can be naturally extenbded to  $H^*$-modules.
\end{prop}

\begin{proof}
The first statement is probably well-known, for completenes we exhibit the proof:
First recall that if $K$ is a finite dimensional Hopf algebra and $M,N$ are $K$-modules
(e.g. $K=H^*$ if $H$ is finite dimensional and $M,N\in\M^H$) then the 
standard action of an element
$x\in K$ in a map $f$, acting on an alement $m\in M$ is given by
\[
(x\cdot f)(m):=x_1f(S(x_2)\cdot m)
\]
If $K=H^*$ and $M,N\in\M^H$ then the above 
formula is
\[
(x\cdot f)(m)=x_1\cdot f((S(x_2)(m_1)) m_0)
=x_2(S(m_1) )x_1\cdot f(m_0)
=x_1( f(m_0)_1) x_2(S(m_1) )f(m_0)_0
\]
\[=x( f(m_0)_1 S(m_1) )f(m_0)_0
\]
and the last term in the equality make sense for $x \in H^*$, independently on the
dimension of $H$, so one {\em defines } the $H^*$-action of $x\in H^*$ on $f:M\to N$ via
\[
(x\cdot f)(m)
:=x( f(m_0)_1 S(m_1) )f(m_0)_0
\]
In other words,
\[x\cdot f=(1\ot m_H^*(x))(\rho_N\ot 1)(f\ot S)\rho_M\]
One can proof by standard diagramatic methods that this is an action, and $f$ is
$H$-colinear (if and only if it is $H^*$-linear)  if and only if
 \[
x\cdot f =\epsilon(x)f=x(1)f\ \forall x\in H^*
\]
Concering the second statement, if $W$ is an $H^*$-module, one may define
\[
W^{H^*}=\{w\in W : x\cdot w=x(1)w\}\cong\Hom_H(k,W)
\]
If $W$ is a right $H$-comodule then it is clear that 
$W^{co H}=\{w : \rho(w)=w\ot 1\}=W^{H ^*}$, so one can extend the definition of
$\H_0^H$ on ${}_{H^*}\M$  simply by
\[
\H_0^H(W):=\frac{W^{H^*}}{\Int\cdot W}
\]
If $W=\Hom_A(M,N)$ then
$W^{H^*}$=$A$ linear and $H^*$-linear maps =$\Hom_A^H(M,N)$, and a definition
of "chain maps up to homotopy" is available definig
\[
\H_0^H(\Hom_A(M,N))=\frac{\Hom_A^H(M,N)}{\Int\cdot \Hom_A(M,N)}
\]
This recover the definition given in \cite{Qi} for finite dimensional Hopf algebras and when
$M$ and $N$ are $\Z$-graded vector spaces, but we emphasizes that this 
definition makes sense in full generality for $H$ a co-Frobenius algebra 
(whose coradical is not necesarily finite over $k[\Z]$).
\end{proof}

A warning on the notation in \cite{Qi}, we call $\H_0^H$ what he calls $\H$ in 
the ungraded case. He defines $\H_n$ only in the graded case but using the degree shifting,
 and not the triangulated structure, so $\H_n$ in \cite{Qi} is different
from our $\H_n^H$.

\end{document}